\documentclass[a4paper, 11pt]{article}

\usepackage[T1]{fontenc}
\usepackage[english]{babel}
\usepackage{amsthm,amsmath,amssymb}
\usepackage[margin=3cm]{geometry}
\usepackage{esint}
\usepackage{eucal}
\usepackage[font=footnotesize]{caption}
\usepackage{listings}
\usepackage{tikz}
\usetikzlibrary{arrows,decorations.markings}
 \tikzstyle{vecArrow} = [thick, decoration={markings,mark=at position
   1 with {\arrow[semithick]{open triangle 60}}},
   double distance=2.5pt, shorten >= 5.5pt, shorten <= .5pt,
   preaction = {decorate}]

\usepackage{algpseudocode}

 \newtheorem{theorem}{Theorem}
 \newtheorem{lemma}[theorem]{Lemma}

\newtheorem{assumption}[theorem]{Assumption}
 \theoremstyle{remark}
 \newtheorem{remark}[theorem]{Remark}
\theoremstyle{definition}
\newtheorem{algorithm}[theorem]{Algorithm}
\newtheorem{definition}{Definition}

 \numberwithin{theorem}{section}
 \numberwithin{equation}{section}

 \newcommand{\step}[1]{\par\noindent{\em Step #1.}}
 \newcommand{\ddiv}{\operatorname{div}}
 \newcommand{\tri}{\mathcal{T}}
 \newcommand{\edges}{\mathcal{E}}
 \newcommand{\Curl}{\operatorname{Curl}}
 \newcommand{\curl}{\operatorname{curl}}
 \newcommand{\Sym}{\mathbb{S}}
 \newcommand{\sym}{\operatorname{sym}}

 \newcommand{\R}{\mathbb{R}}

 \newcommand{\NC}{\mathrm{NC}}

\begin{document}

\author{M. Schedensack\thanks{Institut f\"ur Numerische Simulation, 
            Universit\"at Bonn, Wegelerstra{\ss}e 6, D-53115 Bonn, Germany
             }
} 
\title{A new discretization for $\lowercase{m}$th-Laplace equations\\ with arbitrary
       polynomial degrees\thanks{%
        This work was supported by the Berlin Mathematical School.}
          }
\date{}
 \maketitle
 
\begin{abstract}
This paper introduces new mixed formulations and discretizations for
$m$th-Laplace equations of the form $(-1)^m\Delta^m u=f$ 
for arbitrary $m=1,2,3,\dots$ based on novel Helmholtz-type 
decompositions for tensor-valued functions. 
The new discretizations allow for ansatz spaces of arbitrary polynomial degree and 
the lowest-order choice coincides with the non-conforming FEMs of Crouzeix and 
Raviart for $m=1$ and of Morley for $m=2$.
Since the derivatives are directly approximated, the lowest-order 
discretizations consist of piecewise affine and piecewise constant 
functions for any $m=1,2,\dots$
Moreover, a uniform implementation for arbitrary $m$ is possible. 
Besides the a~priori and a~posteriori analysis, this paper proves optimal 
convergence rates for adaptive algorithms for the new discretizations.
\end{abstract}
 
\noindent
{\small\textbf{Keywords}
$m$th-Laplace equation, polyharmonic equation, non-conforming FEM,
mixed FEM, adaptive FEM, optimality
}

\noindent
{\small\textbf{AMS subject classification}
31A30, 35J30,
65N30, 65N12, 74K20
}

\section{Introduction}

This paper considers $m$th-Laplace equations of the form
\begin{align}\label{e:introHOP}
  (-1)^m\Delta^m u=f 
\end{align}
for arbitrary $m=1,2,3,\dots$
Standard conforming FEMs require ansatz spaces in $H^m_0(\Omega)$. To circumvent 
those high regularity requirements and resulting complicated finite elements,
non-standard methods are of high interest \cite{Morley1968,
EngelGarikipatiHughesLarsonMazzeiTaylor2002,Brenner2012,GudiNeilan2011}.
The novel Helmholtz decomposition of this paper
decomposes any (tensor-valued) $L^2$ function in an $m$th derivative and a 
symmetric part of a Curl. 
Given a tensor-valued function $\varphi$ which satisfies 
$-\ddiv^m\varphi=f$ in the weak sense, the $L^2$ projection of $\varphi$ to the 
space $D^m H^m_0(\Omega)$ of $m$th derivatives then coincides with the 
$m$th derivative of the exact solution of~\eqref{e:introHOP} 
(see Theorem~\ref{t:HOPexistence} below).
This results in novel mixed formulations and 
discretizations for \eqref{e:introHOP}.
This approach generalises the discretizations of~\cite{Schedensack2015,Schedensack2016}
from $m=1$ to $m\geq 1$.

The direct approximation of $D^m u$ instead of $u$ 
enables low order discretizations; only first derivatives appear in the 
symmetric part of the Curl and so the lowest order approach only requires piecewise 
affine functions for any $m$. In contrast to that, even interior penalty methods 
require piecewise quadratic \cite{Brenner2012} resp.\ piecewise cubic 
\cite{GudiNeilan2011} functions for $m=2$ resp.\ $m=3$.
Mnemonic diagrams in Figure~\ref{f:introTriangles} illustrate lowest-order standard 
conforming FEMs from \cite{Zenisek1970} and the lowest-order novel FEMs proposed in this 
work for $m=2,3$.
Since the proposed new FEMs differ only in the number of components in the ansatz spaces, 
an implementation of one single program, which runs for arbitrary order, is possible.
In particular, the system matrices are obtained by integration of standard FEM 
basis functions.

\begin{figure}
\begin{center}
\begin{tikzpicture}[x=0.75cm, y=0.75cm]
 \phantom{\draw (1,1.5) circle (10pt);}
 \node () at (-2.5,0.5){$m=2$};
 \draw (0,0)--(2,0)--(1,1.5)--cycle;
 \fill (0,0) circle (2pt);
 \fill (2,0) circle (2pt);
 \fill (1,1.5) circle (2pt);
 \draw (0,0) circle (4pt);
 \draw (2,0) circle (4pt);
 \draw (1,1.5) circle (4pt);
 \draw (0,0) circle (6pt);
 \draw (2,0) circle (6pt);
 \draw (1,1.5) circle (6pt);
 \draw[->] (1,0)->(1,-0.5);
 \draw[->] (1.5,0.75)->(2,1.083333);
 \draw[->] (0.5,0.75)->(0,1.083333);
\end{tikzpicture}
\hfill
\begin{tikzpicture}[x=0.75cm, y=0.75cm]
 \phantom{\draw[->] (1,0)->(1,-0.5);}
 \node () at (-0.5,1){$3\times$};
 \draw (0,0)--(2,0)--(1,1.5)--cycle;
 \fill (1,0.5) circle (2pt);
 \node () at (3.2,1){$2\times$};
 \draw (3.7,0)--(5.7,0)--(4.7,1.5)--cycle;
 \fill (3.7,0) circle (2pt);
 \fill (5.7,0) circle (2pt);
 \fill (4.7,1.5) circle (2pt);
\end{tikzpicture}\\
\begin{tikzpicture}[x=0.75cm, y=0.75cm]
 \node () at (-2.5,0.5){$m=3$};
 \draw (0,0)--(2,0)--(1,1.5)--cycle;
 \fill (0,0) circle (2pt);
 \fill (2,0) circle (2pt);
 \fill (1,1.5) circle (2pt);
 \fill (1,0.5) circle (2pt);
 \draw (0,0) circle (4pt);
 \draw (2,0) circle (4pt);
 \draw (1,1.5) circle (4pt);
 \draw (0,0) circle (6pt);
 \draw (2,0) circle (6pt);
 \draw (1,1.5) circle (6pt);
 \draw (0,0) circle (8pt);
 \draw (2,0) circle (8pt);
 \draw (1,1.5) circle (8pt);
 \draw (0,0) circle (10pt);
 \draw (2,0) circle (10pt);
 \draw (1,1.5) circle (10pt);
 \draw[->] (1,0)->(1,-0.5);
 \draw[->] (1.5,0.75)->(2,1.083333);
 \draw[->] (0.5,0.75)->(0,1.083333);
 \draw[->] (0.66666,0)->(0.66666,-0.4);
 \draw[->] (0.66666,-0.4)->(0.66666,-0.5);
 \draw[->] (1.66666,0.5)->(2.066666, 0.766666);
 \draw[->] (2.066666, 0.766666)->(2.166666, 0.83333);
 \draw[->] (0.33333,0.5)->(-0.0666666,0.766666);
 \draw[->] (-0.0666666,0.766666)->(-0.166666,0.83333);
 \draw[->] (1.33333,0)->(1.33333,-0.4);
 \draw[->] (1.33333,-0.4)->(1.33333,-0.5);
 \draw[->] (1.33333,1)->(1.73333,1.266666);
 \draw[->] ( 1.73333,1.266666)->(1.83333,1.333333);
 \draw[->] (0.666666,1)->(0.266666,1.266666);
 \draw[->] (0.266666,1.266666)->( 0.166666,1.333333);
\end{tikzpicture}
\hfill
\begin{tikzpicture}[x=0.75cm, y=0.75cm]
 \phantom{\draw[->] (1,0)->(1,-0.5);}
 \node () at (-0.5,1){$4\times$};
 \draw (0,0)--(2,0)--(1,1.5)--cycle;
 \fill (1,0.5) circle (2pt);
 \node () at (3.2,1){$3\times$};
 \draw (3.7,0)--(5.7,0)--(4.7,1.5)--cycle;
 \fill (3.7,0) circle (2pt);
 \fill (5.7,0) circle (2pt);
 \fill (4.7,1.5) circle (2pt);
\end{tikzpicture}
\end{center}
\caption{\label{f:introTriangles}Lowest order standard conforming 
 \cite{Ciarlet1978,Zenisek1970} and novel 
 FEMs for the problem $(-1)^m\Delta^m u=f$ for $m=2,3$.}
\end{figure}
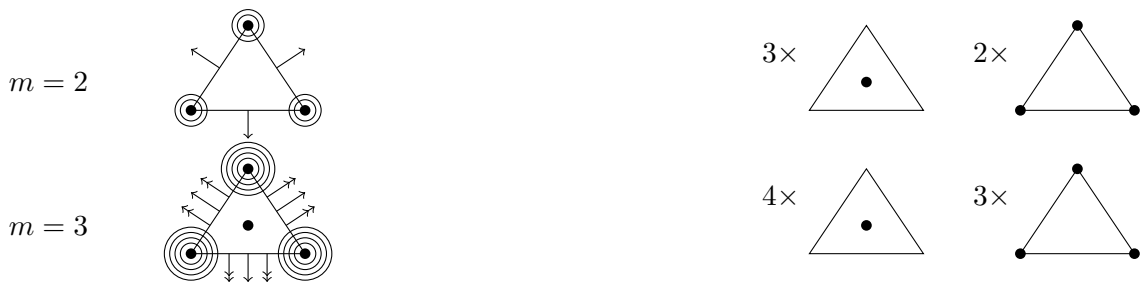

For $m=1,2$ and the lowest polynomial degree in the ansatz spaces, 
discrete Helmholtz decompositions of 
\cite{ArnoldFalk1989, CarstensenGallistlHu2014} prove that 
the discrete solutions are piecewise gradients (resp.\ Hessians) 
of Crouzeix-Raviart \cite{CrouzeixRaviart1973}
(resp.\ Morley \cite{Morley1968}) finite element functions and therefore the new 
discretizations can be regarded as a generalization of those 
non-conforming FEMs to higher polynomial degrees and higher-order problems.
The generalization of~\cite{WangXu2013} of the 
non-conforming Crouzeix-Raviart and Morley FEMs to $m\geq 3$ is restricted 
to a space dimension $\geq m$.

In the context of the novel (mixed) formulations, the discretizations 
appear to be conforming. The new 
generalization to higher polynomial degrees proposed in this paper appears to be 
natural in the sense that the inherent properties 
of the lowest order discretization carry over to higher polynomial ansatz 
spaces, namely an inf-sup condition, the conformity of the method, and a 
crucial projection property (also known as integral mean property of the 
non-conforming interpolation operator).

Besides the a~priori and a~posteriori error analysis, this paper proves optimal 
convergence rates for an adaptive algorithm, which are also observed in 
the numerical experiments from Section~\ref{s:HOPnumerics}.

The remaining parts of this paper are organised as follows.
Section~\ref{s:HOPnotation} introduces some notation
while some preliminary results are proved in Section~\ref{s:HOPprelimresults}.
The proposed discretization of \eqref{e:introHOP} in 
Section~\ref{s:HOPformDiscrete} is based on a novel 
Helmholtz decomposition for higher derivatives which is stated and proved in 
Section~\ref{s:HOPhelmholtzdecomposition}.  
Section~\ref{s:HOPafem} introduces an adaptive algorithm and 
proves optimal convergence rates. Section~\ref{s:HOPnumerics} concludes the 
paper with numerical experiments on fourth- and sixth-order problems.

Throughout this paper, let $\Omega\subseteq\R^2$ be a bounded, polygonal, simply connected
Lip\-schitz domain.
Standard notation on Lebesgue and Sobolev
spaces and their norms is employed with $L^2$ scalar product 
$(\bullet,\bullet)_{L^2(\Omega)}$. Given a Hilbert space $X$, 
let $L^2(\Omega;X)$ resp.\ $H^k(\Omega;X)$ denote the space of functions with 
values in $X$ whose components are in $L^2(\Omega)$ resp.\ $H^k(\Omega)$. 
The space of infinitely differentiable 
functions reads $C^\infty(\Omega)$ and the subspace of functions 
with compact support in $\Omega$ is denoted with $C^\infty_c(\Omega)$.
The piecewise action of differential operators
is denoted with a subscript $\NC$.
The formula $A\lesssim B$ represents an inequality $A\leq CB$ 
for some mesh-size independent, positive generic constant
$C$; $A \approx B$ abbreviates $A \lesssim B \lesssim A$.
By convention, all generic constants $C\approx 1$ do neither
depend on the mesh-size nor on the level of a triangulation
but may depend on the fixed coarse 
triangulation $\tri_0$ and its interior angles.

\section{Notation}\label{s:HOPnotation}

This section introduces notation related to 
higher-order tensors and tensor-valued functions and triangulations.

Define the set of $\ell$-tensors over $\R^2$ by 
\begin{equation*}
  \mathbb{X}(\ell):= \begin{cases}
                       \R & \text{ for }\ell=0,\\
                       \prod_{j=1}^\ell \R^2
                            =\R^2\times\dots\times\R^2\cong\R^{2^\ell}
                                      & \text{ for }\ell\geq 1
                     \end{cases}
\end{equation*}
and let $\mathfrak{S}_\ell:=\{\sigma:\{1,\dots,\ell\}\to\{1,\dots,\ell\}\mid 
\sigma\text{ is bijective}\}$ denote the symmetric group, 
i.e., the set of all permutations of $(1,\dots,\ell)$. 
Define the set of symmetric tensors 
$\Sym(\ell)\subseteq \mathbb{X}(\ell)$ by
\begin{align*}
\Sym(\ell) :=\{A\in\mathbb{X}(\ell)\mid 
      \forall (j_1,\dots,j_\ell)\in\{1,2\}^\ell\,
      \forall\sigma\in \mathfrak{S}_\ell:\;
       A_{j_1,\dots,j_\ell} = A_{j_{\sigma(1)},\dots,j_{\sigma(\ell)}}\}. 
\end{align*}
The symmetric part $\sym A\in \Sym(\ell)$ of a tensor $A\in\mathbb{X}(\ell)$
is defined by 
\begin{align*}
 (\sym A)_{j_1,\dots,j_\ell}
   :=(\mathrm{card}(\mathfrak{S}_\ell))^{-1} 
     \sum_{\sigma\in \mathfrak{S}_\ell} A_{j_{\sigma(1)},\dots,j_{\sigma(\ell)}} 
\end{align*}
for all $(j_1,\dots,j_{\ell})\in\{1,2\}^{\ell}$, where $\mathrm{card}(\mathcal{M})$ 
denotes the number of elements in a set $\mathcal{M}$.
For $\ell=2$, the set $\Sym(2)$ coincides with the set of symmetric $2\times 2$ 
matrices, while for $\ell=3$, the tensors $A\in\Sym(3)$ consist of the four 
different components $A_{111}$, $A_{112}=A_{121}=A_{211}$, 
$A_{122}=A_{212}=A_{221}$, and $A_{222}$.
Given $\ell$-tensors $A,B\in \mathbb{X}(\ell)$ and a vector $q\in\R^2$, define the 
scalar product $A:B\in\R$ and the dot product $A\cdot q\in \mathbb{X}(\ell-1)$ by
\begin{align*}
 A:B&:=\sum_{(j_1,\dots,j_\ell)\in\{1,2\}^\ell} A_{j_1,\dots,j_\ell} B_{j_1,\dots,j_\ell},\\
 (A\cdot q)_{j_1,\dots,j_{\ell-1}} 
   &:= A_{j_1,\dots,j_{\ell-1},1}\; q_1
     + A_{j_1,\dots,j_{\ell-1},2}\; q_2
\end{align*}
for all $(j_1,\dots,j_{\ell-1})\in\{1,2\}^{\ell-1}$.
The following definition summarizes some differential operators.
Recall that, for a Hilbert space $X$, the space $H^1(\Omega;X)$ (resp.\
$L^2(\Omega;X)$) denotes the space of $H^1$ (resp.\ $L^2$) functions with 
components in $X$.

\begin{definition}[differential operators]\label{d:HOPdiff}
Let $v\in H^\ell_0(\Omega)$ and $\sigma\in H^1(\Omega;\mathbb{X}(\ell))$ and define 
$\mathfrak{p}:\{1,2\}\to\{1,2\} $ by $\mathfrak{p}(1)=2$ and 
$\mathfrak{p}(2)=1$. Define the 
$\ell$th derivative $D^\ell v\in L^2(\Omega;\mathbb{X}(\ell))$ of $v$, the derivative 
$D \sigma\in L^2(\Omega;\mathbb{X}(\ell+1))$, the divergence 
$\ddiv \sigma\in L^2(\Omega;\mathbb{X}(\ell-1))$, the Curl,
$\Curl \sigma\in L^2(\Omega;\mathbb{X}(\ell+1))$, and the curl, 
$\curl \sigma\in L^2(\Omega;\mathbb{X}(\ell-1))$ by
\begin{align*}
(D^\ell v)_{j_1,\dots,j_\ell}
   &:=\partial^{\ell} v/(\partial x_{j_1}\dots\partial x_{j_\ell}) ,\\
 (D \sigma)_{j_1,\dots,j_{\ell+1}}
   &:= \partial \sigma_{j_1,\dots,j_{\ell}}/\partial x_{j_{\ell+1}},\\
    (\Curl \sigma)_{j_1,\dots,j_{\ell+1}}
   &:= (-1)^{j_{\ell+1}}
         \partial \sigma_{j_1,\dots,j_{\ell}}/ \partial x_{\mathfrak{p}(j_{\ell+1})},
                   \\
 (\ddiv \sigma)_{j_1,\dots,j_{\ell-1}} 
  &:= \partial \sigma_{j_1,\dots,j_{\ell-1},1}/ \partial x_1 
     +\partial \sigma_{j_1,\dots,j_{\ell-1},2}/\partial x_2, \\
  (\curl \sigma)_{j_1,\dots,j_{\ell-1}} 
    &:= -\partial \sigma_{j_1,\dots,j_{\ell-1},1} /\partial x_2 
    +  \partial \sigma_{j_1,\dots,j_{\ell-1},2}/\partial x_1  
\end{align*}
for $(j_1,\dots,j_{\ell+1})\in\{1,2\}^{\ell+1}$ .
\end{definition}

For $\ell=2$, these definitions coincide with the row-wise application 
of $D$, $\ddiv$, $\Curl$, and $\curl$.
The $L^2$ scalar product $(\bullet,\bullet)_{L^2(\Omega)}$ of tensor-valued 
functions $f,g:\Omega\to\mathbb{X}(\ell)$ is defined by 
$(f,g)_{L^2(\Omega)}:=\int_\Omega f:g\,dx$.
Given $\psi\in L^2(\Omega;\mathbb{X}(\ell))$ such that there exists 
$g\in L^2(\Omega)$ with 
\begin{align*}
  (\psi,D^\ell v)_{L^2(\Omega)} = (-1)^\ell \,(g, v)_{L^2(\Omega)}
  \qquad\text{for all }v\in H^\ell_0(\Omega),
\end{align*}
define the $\ell$th order divergence $\ddiv^\ell\psi:=g$ of $\psi$.
The space $H(\ddiv^\ell,\Omega)\subseteq L^2(\Omega;\mathbb{X}(\ell))$ 
is defined by
\begin{align*}
 H(\ddiv^\ell,\Omega):=\{\psi\in L^2(\Omega;\mathbb{X}(\ell))\mid 
     \ddiv^\ell\psi\in L^2(\Omega)\}.
\end{align*}
Define furthermore for $k\geq\ell$
\begin{align*}
  H(\ddiv^\ell,\Omega;\mathbb{X}(k))
   :=\left\{\psi\in L^2(\Omega;\mathbb{X}(k)) \left\vert
       \begin{array}{l}
       \forall (j_1,\dots,j_{k-\ell})\in\{1,2\}^{k-\ell}:\\
      \psi_{j_1,\dots,j_{k-\ell},\bullet}\in H(\ddiv^\ell,\Omega)  
       \end{array}
       \right\}\right..
\end{align*}

\begin{remark}
Note that the existence of the $\ell$th weak divergence does not imply the 
existence of any $k$-th divergence for $1\leq k\leq \ell$, e.g.,   
$H(\ddiv,\Omega;\mathbb{X}(\ell))\not\subseteq H(\ddiv^\ell,\Omega)$
for $\ell>1$.
\end{remark}

A shape-regular triangulation $\tri$ of a bounded, polygonal, open
Lipschitz domain $\Omega\subseteq \mathbb{R}^2$ is a set of closed triangles $T\in \tri$
such that $\overline{\Omega}=\bigcup\tri$ and any two distinct 
triangles are either disjoint or share exactly 
one common edge or one vertex.
Let $\edges(T)$ denote the edges of a triangle $T$
and $\edges:=\edges(\tri):=\bigcup_{T\in\tri} \edges(T)$ the set of edges 
in $\tri$. 
Any edge $E\in\edges$ is associated with a fixed orientation of the unit normal 
$\nu_E$ on $E$ (and $\tau_E=(0,-1;1,0)\nu_E$ denotes the unit tangent on $E$). 
On the boundary, $\nu_E$ is the outer unit normal of $\Omega$,
while for interior edges $E\not\subseteq\partial\Omega$, the orientation is 
fixed through the choice of the triangles $T_+\in\tri$ and $T_-\in\tri$ with 
$E=T_+\cap T_-$ and $\nu_E:=\nu_{T_+}\vert_E$ is the outer normal of $T_+$ on $E$.
In this situation, $[v]_E:=v\vert_{T_+}-v\vert_{T_-}$ denotes the jump across 
$E$. For an edge $E\subseteq\partial\Omega$ on the boundary, the jump across 
$E$ reads $[v]_E:=v$.
For $T\in\tri$ and $X\subseteq \mathbb{X}(\ell)$, let 
\begin{align*}
P_k(T;X) &:= \left\{v:T\rightarrow X \left\vert 
          \text{ each component}\text{ of }v
    \text{ is a polynomial of total degree}\leq k
   \right\}\right.;\\
P_k(\tri;X) &:= \{v:\Omega\rightarrow X\;| \;
\forall T\in\tri:\;v|_T \in P_k(T;X)\}
\end{align*}
denote the set of piecewise polynomials and $P_k(\tri):=P_k(\tri;\R)$. 
Given a subspace $X\subseteq L^2(\Omega;\mathbb{X}(\ell))$, 
let $\Pi_{X}:L^2(\Omega;\mathbb{X}(\ell))\to X$
denote the $L^2$ projection onto $X$ and let $\Pi_k$ abbreviate 
$\Pi_{P_k(\tri;\mathbb{X}(\ell))}$.
Given a triangle $T\in\tri$, let $h_T:=(\mathrm{meas}_2(T))^{1/2}$ 
denote the square root of the area of 
$T$ and 
let $h_\tri\in P_0(\tri)$ denote the piecewise constant mesh-size 
with $h_\tri\vert_T:=h_T$ for all $T\in\tri$. 
For a set of triangles $\mathcal{M}\subseteq\tri$,
let $\|\bullet\|_{\mathcal{M}}$ abbreviate 
\begin{align*}
 \|\bullet\|_{\mathcal{M}}:=\sqrt{\sum_{T\in\mathcal{M}}
         \|\bullet\|_{L^2(T)}^2}.
\end{align*}

\section{Results for tensor-valued functions}\label{s:HOPprelimresults}

The main result of this section is Theorem~\ref{t:HOPcurllesssymcurl2}, which 
proves that $\left\|\sym\Curl\bullet\right\|_{L^2(\Omega)}$ defines a norm on 
the space $Y$ defined in~\eqref{e:HOPdefY} below and can, thus, be viewed as a 
generalized Korn inequality.
The following theorem is used in the proof of Theorem~\ref{t:HOPcurllesssymcurl2}.
Recall the definition of the Curl and the symmetric part of a tensor 
from Section~\ref{s:HOPnotation}.

\begin{theorem}\label{t:HOPcurllesssymcurl1}
Any $\gamma\in H^1(\Omega;\Sym(m-1))$ satisfies 
\begin{align*}
 \left\|\Curl\gamma\right\|_{L^2(\Omega)}
  \lesssim\left\|\sym\Curl\gamma\right\|_{L^2(\Omega)}
   + \left\|\gamma\right\|_{L^2(\Omega)}.
\end{align*}
\end{theorem}

\begin{proof}
The proof is subdivided in three steps.

\step{1}
Let $0\leq k\leq m$ and $\mathbf{j}(k)=(j_1,\dots,j_m)\in \{1,2\}^m$ 
with $j_\ell=1$ for all $\ell\in\{1,\dots,k\}$ and $j_\ell=2$ for all
$\ell\in\{k+1,\dots,m\}$, i.e.,
\begin{align*}
 \mathbf{j}(k)=(\underbrace{1,\dots,1}_k,\underbrace{2,\dots,2}_{m-k}).
\end{align*}
The combination of the definitions of 
$\sym$ and $\Curl$ reads 
\begin{align}\label{e:HOPproofcurllesssymcurl1}
 (\sym\Curl\gamma)_{\mathbf{j}(k)} = \mathrm{card}(\mathfrak{S}_{m})^{-1} 
      \sum_{\sigma\in\mathfrak{S}_{m}} (-1)^{j_{\sigma(m)}}
         \frac{\partial}{\partial x_{\mathfrak{p}(j_{\sigma(m)})}} 
            \gamma_{j_{\sigma(1)},\dots,j_{\sigma(m-1)}}.
\end{align}
Let $\overline{\mathbf{j}}(k):=(j_1,\dots,j_{m-1})\in\{1,2\}^{m-1}$ 
be the multi-index with the same number 
of ones and the number of twos reduced by one and 
$\underline{\mathbf{j}}(k):= (j_2,\dots,j_{m})\in\{1,2\}^{m-1}$ 
the multi-index with the same number 
of twos and the number of ones reduced by one, i.e.,
\begin{align*}
 \overline{\mathbf{j}}(k)=(\underbrace{1,\dots,1}_k,\underbrace{2,\dots,2}_{m-k-1})
 \qquad\text{and}\qquad
 \underline{\mathbf{j}}(k)=(\underbrace{1,\dots,1}_{k-1},\underbrace{2,\dots,2}_{m-k}).
\end{align*}
The symmetry of $\gamma$ implies that 
$\gamma_{j_{\sigma(1)},\dots,j_{\sigma(m-1)}}=\gamma_{\underline{\mathbf{j}}(k)}$
if $j_{\sigma(m)}=1$ and 
$\gamma_{j_{\sigma(1)},\dots,j_{\sigma(m-1)}}=\gamma_{\overline{\mathbf{j}}(k)}$
if $j_{\sigma(m)}=2$.
Since the number of permutations $\sigma\in\mathfrak{S}_m$ such that 
$j_{\sigma(m)}=1$ is 
$k \,\mathrm{card}(\mathfrak{S}_{m-1})$ and 
the number of permutations $\sigma\in\mathfrak{S}_m$ such that $j_{\sigma(m)}=2$ is 
$(m-k)\,\mathrm{card}(\mathfrak{S}_{m-1})$
and since $\mathrm{card}(\mathfrak{S}_{m})=m!$ and 
$\mathrm{card}(\mathfrak{S}_{m-1})=(m-1)!$, 
this implies that 
\eqref{e:HOPproofcurllesssymcurl1} equals 
\begin{equation}\label{e:HOPcurllesssymcurlProof2}
\begin{aligned}
& (\sym\Curl\gamma)_{\mathbf{j}(k)} \\
   &\quad 
 =  \frac{\mathrm{card}(\mathfrak{S}_{m-1})}{\mathrm{card}(\mathfrak{S}_{m})} 
 \left( (m-k)\,
           \frac{\partial \gamma_{\overline{\mathbf{j}}(k)}}{\partial x}
         - k\,
             \frac{\partial \gamma_{\underline{\mathbf{j}}(k)}}{\partial y}
\right)
 = \frac{m-k}{m}\, \frac{\partial \gamma_{\overline{\mathbf{j}}(k)}}{\partial x}
  - \frac{k}{m}\,  \frac{\partial \gamma_{\underline{\mathbf{j}}(k)}}{\partial y}.
\end{aligned}
\end{equation}

\step{2}
This step applies \cite[Chap.~3, Thm.~7.6]{Necas1967} 
and \cite[Chap.~3, Thm.~7.8]{Necas1967} to operators $N_k$ defined 
below. Step~3 then proves a relation between these operators and the 
operator $\sym\Curl$.

Define for $k\in\{1,\dots,m+1\}$, $s\in\{1,\dots,m\}$, and a multi-index 
$\kappa\in\{(1,0),(0,1)\}$
\begin{align*}
 a_{k,s,\kappa}:=\begin{cases}
                   -(k-1)/m & \text{ if }s=k-1\text{ and }\kappa=(0,1),\\
                   (m-k+1)/m & \text{ if }s=k\text{ and }\kappa=(1,0),\\
                   0 & \text{ else}.
                 \end{cases}
\end{align*}
Furthermore, define for $\xi\in\R^2$ 
\begin{align*}
  \mathfrak{N}_{ks}\xi:= \sum_{\kappa=(1,0),(0,1)} a_{k,s,\kappa} \xi^\kappa
           =
             \begin{cases}
               -(k-1) \xi_2/m & \text{ if }s=k-1,\\
               (m-k+1) \xi_1/m   & \text{ if }s=k,\\
               0               & \text{else}
             \end{cases}
\end{align*}
with the multi-index notation 
$\xi^\kappa=\xi_1^{\kappa_1}\xi_2^{\kappa_2}$.
Then the matrix $(\mathfrak{N}_{ks}\xi)_{\substack{1\leq k\leq (m+1)\\
                     1\leq s\leq m}}$ reads
\begin{align*}
 \frac{1}{m}
   \begin{pmatrix}
     m\xi_1 & 0 & 0 & 0 & \dots & 0\\
     -\xi_2 & (m-1) \xi_1 & 0 & 0 & \dots & 0\\
      0 & -2 \xi_2 & (m-2) \xi_1 & 0 & \dots & 0\\
      \vdots &  &        &        & & \vdots       \\
      \vdots &  & \ddots & \ddots & & \vdots \\
      \vdots &  &        &        & & \vdots       \\
      0 & \dots & 0 & -(m-2)\xi_2 & 2\xi_1 & 0 \\
      0 & \dots &  & 0 & -(m-1) \xi_2 & \xi_1\\
      0 & \dots &  &  & 0 & -m\xi_2
   \end{pmatrix}
  \in\R^{(m+1)\times m}.
\end{align*}
If $\xi\neq 0$, the columns of this matrix are linear independent. 
Define the operators $(N_k)_{k=1,\dots,m+1}$,
$N_k:H^1(\Omega;\R^m)\to L^2(\Omega)$, by
\begin{align*}
 N_k v:= \sum_{s=1}^{m} \sum_{\kappa=(1,0),(0,1)} a_{k,s,\kappa} D^\kappa v_s.
\end{align*}
Then, the combination of \cite[Chap.~3, Thm.~7.6]{Necas1967} 
with \cite[Chap.~3, Thm.~7.8]{Necas1967}
proves 
\begin{align}\label{e:HOPcurllesssymcurlProof3}
 \sum_{s=1}^{m} \|v_s\|_{H^1(\Omega)}^2
   \lesssim \sum_{k=1}^{m+1} \|N_k v\|_{L^2(\Omega)}^2
      + \sum_{s=1}^m \|v_s\|_{L^2(\Omega)}^2.
\end{align}

\step{3}
This step proves a relation between $(N_k)_{k=1,\dots,m+1}$ and $\sym\Curl$ 
for a proper choice of $v=(v_1,\dots,v_m)$.

Define $v=(v_1,\dots,v_m)\in H^1(\Omega;\R^m)$ by setting for each $s\in\{1,\dots,m\}$
the function $v_s:=\gamma_{\ell_1,\dots,\ell_{m-1}}$ with $\ell_1=\dots=\ell_{s-1}=1$
and $\ell_s=\dots=\ell_{m-1}=2$
(with $(\ell_1,\dots,\ell_{m-1})=(2,\dots,2)$ for $s=1$ and 
$(\ell_1,\dots,\ell_{m-1})=(1,\dots,1)$ for $s=m$).
The symmetry of $\gamma$ proves 
\begin{equation}\label{e:HOPcurllesssymcurlProof4}
\begin{aligned}
  \left\|\Curl\gamma\right\|_{L^2(\Omega)}^2
    &\lesssim \sum_{s=1}^{m} \|v_s\|_{H^1(\Omega)}^2
 \qquad\text{and}\qquad
  \sum_{s=1}^m \|v_s\|_{L^2(\Omega)}^2 
    &\approx\|\gamma\|_{L^2(\Omega)}^2.
\end{aligned}
\end{equation}
With the notation from Step~1 it holds that 
$v_s=\gamma_{\overline{\mathbf{j}}(s-1)}=\gamma_{\underline{\mathbf{j}}(s)}$
and the definition of $N_{k}$ from Step~2 
and \eqref{e:HOPcurllesssymcurlProof2} reveal
\begin{align*}
   N_{k+1} v = 
   (m-k)/m\,(\partial v_{k+1}/\partial x)
   - k/m\,(\partial v_{k}/\partial y)
  =(\sym\Curl\gamma)_{\mathbf{j}(k)} .
\end{align*}
This leads to  
\begin{align*}
   \sum_{k=1}^{m+1} \|N_k v\|_{L^2(\Omega)}^2
   \leq \left\|\sym\Curl\gamma\right\|_{L^2(\Omega)}^2.
\end{align*}
This, \eqref{e:HOPcurllesssymcurlProof4}, and an application 
of \eqref{e:HOPcurllesssymcurlProof3} implies the 
assertion.
\end{proof}

Define, for $m\geq 1$, the spaces 
\begin{equation}\label{e:HOPdefY}
\begin{aligned}
 \mathfrak{H}(\Omega,m-1)
  &:=\left\{v\in H^1(\Omega;\Sym(m-1))\left\vert \;
       \textstyle{\int_\Omega v\,dx=0}\right\}\right.,\\
 Z&:=\left\{\beta\in \mathfrak{H}(\Omega,m-1) \left\vert\;
       \sym\Curl\beta = 0\right\}\right.,\\
 Y&:=\left\{\gamma\in \mathfrak{H}(\Omega,m-1) \left\vert \;
        \forall \beta\in Z:\; (\Curl\beta,\Curl\gamma)_{L^2(\Omega)}=0
         \right\}\right..
\end{aligned}
\end{equation}
A computation reveals for $m=2$, that the spaces $Z$ and $Y$ read
\begin{equation}
\begin{aligned}\label{e:HOPZm2}
  Z&=\{\gamma\in \mathfrak{H}(1)\mid \exists c_1\in\R, c_2\in\R^2
     \text{ with }\gamma(x) = c_1 x + c_2\},\\
   Y&=\left\{\gamma\in H^1(\Omega;\R^2)\left\vert 
    \textstyle{\int_\Omega \gamma\,dx=0\text{ and }
     \int_\Omega\ddiv\gamma\,dx=0}\right\}\right.
\end{aligned}
\end{equation}
and for $m=3$ the space $Z$ reads
\begin{align}\label{e:Zform=3}
   Z=\left\{\gamma\in \mathfrak{H}(2) 
       \left\vert 
        \begin{array}{l}
           \exists c_1,c_2,c_3\in\R, c_4\in\R^{2\times 2}
         \text{ with }\\
          \gamma(x,y) = \begin{pmatrix}
                                      c_1 x^2+2c_2 x & c_1 xy + c_2 y + c_3 x\\
                                      c_1 xy+c_2 y+c_3 x & c_1 y^2 + 2c_3 y 
                                    \end{pmatrix}
            + c_4
        \end{array}
\right\}\right..
\end{align}

The following theorem generalizes \cite[Lemma~3.3]{CarstensenGallistlHu2014} from 
$m=2$ to higher-order tensors $m>2$ and states 
that $\left\|\sym\Curl\bullet\right\|_{L^2(\Omega)}$
defines a norm on $Y$.
Note that $\left\|\Curl\bullet\right\|_{L^2(\Omega)}
=\left\| D\bullet\right\|_{L^2(\Omega)}$.

\begin{theorem}\label{t:HOPcurllesssymcurl2}
Any $\gamma\in Y$ satisfies 
\begin{align*}
 \left\|\Curl\gamma\right\|_{L^2(\Omega)} 
   \lesssim \left\|\sym\Curl\gamma\right\|_{L^2(\Omega)}.
\end{align*}
\end{theorem}

\begin{proof}
Assume for contradiction that the statement does not hold. 
Then there exists a sequence $(\gamma_n)_{n\in\mathbb{N}}\in Y^{\mathbb{N}}$ 
with 
\begin{align*}
 n \left\|\sym\Curl\gamma_n\right\|_{L^2(\Omega)}
  \leq  \left\|\Curl\gamma_n\right\|_{L^2(\Omega)} = 1.
\end{align*}
Since $Y\subseteq \mathfrak{H}(m-1)$, Poincar\'e's 
inequality implies that all components of
$\gamma_n$ are bounded in $H^1(\Omega)$. Since 
$H^1(\Omega;\mathbb{X}(m-1))$ is reflexive and compactly embedded in 
$L^2(\Omega;\mathbb{X}(m-1))$, 
there 
exists a subsequence (not relabelled) with a limit
$\gamma\in L^2(\Omega;\mathbb{X}(m-1))$, $\gamma_n\to \gamma$ in 
$L^2(\Omega;\mathbb{X}(m-1))$. This and Theorem~\ref{t:HOPcurllesssymcurl1} 
imply
\begin{align*}
 \left\|\Curl(\gamma_n-\gamma_\ell)\right\|_{L^2(\Omega)}
   &\lesssim \left\|\sym\Curl(\gamma_n-\gamma_\ell)\right\|_{L^2(\Omega)}
              + \|\gamma_n-\gamma_\ell\|_{L^2(\Omega)}\\
   &\leq \frac{1}{n} + \frac{1}{\ell} + \|\gamma_n-\gamma_\ell\|_{L^2(\Omega)}
    \to 0 \qquad\text{as }n,\ell\to\infty.
\end{align*}
The Poincar\'e inequality and the completeness of $H^1(\Omega;\mathbb{X}(m-1))$
imply the existence of $\widetilde{\gamma}\in H^1(\Omega;\mathbb{X}(m-1))$
with $\gamma_n\to\widetilde{\gamma}$ in $H^1(\Omega;\mathbb{X}(m-1))$ and 
thus $\gamma=\widetilde{\gamma}$.
It holds that $\left\|\sym\Curl\bullet\right\|_{L^2(\Omega)}
\leq\left\|\Curl\bullet\right\|_{L^2(\Omega)}$ and, therefore,
$\left\|\sym\Curl\bullet\right\|_{L^2(\Omega)}$ defines a bounded
functional on $H^1(\Omega;\mathbb{X}(m-1))$.
Hence, 
\begin{align}\label{e:HOPcurllesssymcurlProof1}
 \left\|\sym\Curl\gamma\right\|_{L^2(\Omega)} 
   = \lim_{n\to\infty}  \left\|\sym\Curl\gamma_n\right\|_{L^2(\Omega)} =0.
\end{align}
Let $\beta\in Z$. Since $\gamma_n\in Y$, the Cauchy inequality reveals 
\begin{align*}
 (\Curl\beta,\Curl\gamma)_{L^2(\Omega)}
  &= (\Curl\beta,\Curl(\gamma-\gamma_n))_{L^2(\Omega)}\\
  &\leq \left\|\Curl\beta\right\|_{L^2(\Omega)} \;
       \left\|\Curl(\gamma-\gamma_n)\right\|_{L^2(\Omega)}
  \to 0 \qquad\text{as }n\to\infty.
\end{align*}
This and \eqref{e:HOPcurllesssymcurlProof1} lead to $\gamma\in Z\cap Y$ and 
therefore $\gamma=0$. This contradicts 
$\left\|\Curl\gamma\right\|_{L^2(\Omega)}
=\lim_{n\to\infty}\left\|\Curl\gamma_n\right\|_{L^2(\Omega)} = 1$ and, hence, implies the 
assertion.
\end{proof}

\begin{remark}[dependency on the domain]\label{r:HOPdepdomain}
The proof by contradiction from Theorem~\ref{t:HOPcurllesssymcurl2} 
does not provide information about the dependency on the domain. 
A scaling argument reveals that it does not depend on the size 
of the domain, but it may depend on its shape. 
\end{remark}

\section{Helmholtz decomposition for higher orders}
\label{s:HOPhelmholtzdecomposition}

This section proves a Helm\-holtz decomposition 
of $L^2$ tensors into $m$th derivatives and the symmetric part of a Curl
in Theorem~\ref{t:HOPhelmholtz}. This is 
a generalization of the Helmholtz decomposition of
\cite{BeiraodaVeigaNiiranenStenberg2007} for fourth-order problems ($m=2$).
The proof is based on Theorem~\ref{t:HOPhelmholtzCinfty} below, which 
characterizes 
$m$th-divergence-free smooth functions as symmetric parts of Curls.

\begin{theorem}\label{t:HOPhelmholtzCinfty}
Let $m\geq 1$ and $\tau\in C^\infty(\Omega;\Sym(m))$ with $\ddiv^m\tau=0$.
Then there exists $\gamma\in C^\infty(\Omega;\mathbb{X}(m-1))$ with 
\begin{align*}
 \tau = \sym\Curl\gamma.
\end{align*}
\end{theorem}

\begin{proof}
The proof is based on mathematical induction. 
 
The base case $m=1$ is a classical result \cite{Rudin1976}.
Assume as induction hypothesis that the statement holds for $(m-1)$, i.e.,
for all $\widetilde\tau\in C^\infty(\Omega;\Sym(m-1))$ with $\ddiv^{m-1}\widetilde\tau=0$ 
there exists $\gamma\in C^\infty(\Omega;\mathbb{X}(m-2))$ with $\widetilde\tau = 
\sym\Curl\gamma$.

The inductive step is split in five steps.
Suppose that $\tau\in C^\infty(\Omega;\Sym(m))$ with $\ddiv^m\tau=0$.

\step{1}
Then $\ddiv\tau\in C^\infty(\Omega;\mathbb{X}(m-1))$ and  
$\ddiv^{m-1}\ddiv \tau=0$. Let $(j_1,\dots,j_{m-1})\in\{1,2\}^{m-1}$ and 
$\sigma\in\mathfrak{S}_{m-1}$. Recall the definition of the divergence 
from Definition~\ref{d:HOPdiff}.
The symmetry of $\tau$ implies
\begin{align*}
 &(\ddiv\tau)_{j_1,\dots,j_{m-1}} 
   = \partial \tau_{j_1,\dots,j_{m-1},1}/\partial x_1 
     + \partial \tau_{j_1,\dots,j_{m-1},2}/\partial x_2 \\
 &\qquad\qquad
  = \partial \tau_{j_{\sigma(1)},\dots,j_{\sigma(m-1)},1}/\partial x_1 
     + \partial \tau_{j_{\sigma(1)},\dots,j_{\sigma(m-1)},2}/\partial x_2
   = (\ddiv\tau)_{j_{\sigma(1)},\dots,j_{\sigma(m-1)}}.
\end{align*}
Hence, $\ddiv\tau\in C^\infty(\Omega;\Sym(m-1))$.
The induction hypothesis guarantees
the existence of $\beta\in C^\infty(\Omega;\mathbb{X}(m-2))$ with 
$\ddiv\tau=\sym\Curl\beta$.

\step{2}
This step defines some 
$\widehat{\beta}\in C^\infty(\Omega;\mathbb{X}(m))$
with $\ddiv\widehat{\beta}=\ddiv \tau$.

The definitions of $\sym$ and $\Curl$ from Section~\ref{s:HOPnotation}
for tensors combine to 
\begin{equation}\label{e:HOPproofHelmholtzCinftyStep21}
\begin{aligned}
 &(\sym\Curl\beta)_{j_1,\dots,j_{m-1}}\\
  &\qquad\qquad = (\mathrm{card}(\mathfrak{S}_{m-1}))^{-1} 
     \sum_{\sigma\in\mathfrak{S}_{m-1}} (-1)^{j_{\sigma(m-1)}}
         \frac{\partial}{\partial x_{\mathfrak{p}(j_{\sigma(m-1)})}} 
            \beta_{j_{\sigma(1)},\dots,j_{\sigma(m-2)}}.
\end{aligned}
\end{equation}
Define $\widehat{\beta}\in C^\infty(\Omega;\mathbb{X}(m))$
by 
\begin{align}\label{e:HOPproofHelmholtzCinftyStep22}
 \widehat{\beta}_{j_1,\dots,j_m}
  := (-1)^{\mathfrak{p}(j_{m})} (\mathrm{card}(\mathfrak{S}_{m-1}))^{-1} 
                \sum_{\substack{\sigma\in\mathfrak{S}_{m-1}\\
                                  j_{\sigma(m-1)}=\mathfrak{p}(j_m)}}
        \beta_{j_{\sigma(1)},\dots,j_{\sigma(m-2)}}. 
\end{align}
The definition of $\widehat{\beta}$ implies
\begin{align*}
 (\ddiv \widehat{\beta})_{j_1,\dots,j_{m-1}}
   &= (\mathrm{card}(\mathfrak{S}_{m-1}))^{-1} \sum_{k=1}^2 
          (-1)^{\mathfrak{p}(k)}
         \frac{\partial}{\partial x_k}
          \sum_{\substack{\sigma\in\mathfrak{S}_{m-1}\\
                                  j_{\sigma(m-1)}=\mathfrak{p}(k)}} 
              \beta_{j_{\sigma(1)},\dots,j_{\sigma(m-2)}}.
\end{align*}
Since $j_{\sigma(m-1)}=\mathfrak{p}(k)$ if and only if
$\mathfrak{p}(j_{\sigma(m-1)})=k$, this equals
\begin{align*}
 &(\mathrm{card}(\mathfrak{S}_{m-1}))^{-1} 
      \sum_{k=1}^2    
      \sum_{\substack{\sigma\in\mathfrak{S}_{m-1}\\
                       \mathfrak{p}(j_{\sigma(m-1)})=k}} 
               (-1)^{j_{\sigma(m-1)}} 
              \frac{\partial}{\partial x_{\mathfrak{p}(j_{\sigma(m-1)})}}
              \beta_{j_{\sigma(1)},\dots,j_{\sigma(m-2)}}\\
 &\qquad\qquad = 
 (\mathrm{card}(\mathfrak{S}_{m-1}))^{-1} 
          \sum_{\sigma\in\mathfrak{S}_{m-1}} 
               (-1)^{j_{\sigma(m-1)}} 
              \frac{\partial}{\partial x_{\mathfrak{p}(j_{\sigma(m-1)})}}
              \beta_{j_{\sigma(1)},\dots,j_{\sigma(m-2)}}
\end{align*}
and, hence, the combination of the foregoing two displayed 
formulae with \eqref{e:HOPproofHelmholtzCinftyStep21} leads to 
$\ddiv \widehat{\beta} = \sym\Curl\beta$.
The combination with Step~1 proves $\ddiv \widehat{\beta}= \ddiv \tau$.

\step{3} 
Since $\ddiv(\tau-\widehat{\beta})=0$, 
the base case 
(applied ``row-wise'' to $(\tau-\widehat{\beta})_{j_1,\ldots,j_{m-1},\bullet}$)
guarantees the existence 
of $\gamma\in C^\infty(\Omega;\mathbb{X}(m-1))$ with 
$\tau -\widehat{\beta} = \Curl\gamma$.

\step{4} This step shows $\sym(\widehat{\beta})=0$.

Let $(j_1,\dots,j_m)\in\{1,2\}^m$ be fixed and
let $N_1 :=\mathrm{card}(\{k\in\{1,\dots,m\}\mid j_k=1\})$ 
and $N_2 := \mathrm{card}(\{k\in\{1,\dots,m\}\mid j_k=2\})$
be the number of ones and twos.
Then 
\begin{align}\label{e:HOPproofHelmholtzCinftyStep31}
 M_1(j_m):=N_1-(2-j_m)
 \quad\text{and}\quad M_2(j_m):=N_2-(j_m-1)
\end{align}
are the numbers 
of ones and twos in $(j_1,\dots,j_{m-1})$.
Define the index set 
\begin{align*}
 \mathfrak{T}:=\left\{(k_1,\ldots,k_{m-2})\in \{1,2\}^{m-2}\;\left|\;
\sum_{\ell=1}^{m-2} k_\ell = (N_1-1)+2 (N_2-1)\right.\right\}.
\end{align*}
This set $\mathfrak{T}$ contains exactly all indices $(k_1,\ldots,k_{m-2})$ with 
$(N_1-1)$ many ones and $(N_2-1)$ many twos.
Note that $j_{\sigma(m-1)}=\mathfrak{p}(j_m)$
implies that $\{j_{\sigma(m-1)},j_m\}=\{1,2\}$ and
the elements of $\mathfrak{T}$ are the only indices which 
appear as indices of $\beta$ in the sum 
in \eqref{e:HOPproofHelmholtzCinftyStep22}.
For $\mathbf{j}\in \mathfrak{T}$, each $\beta_\mathbf{j}$
appears $M_1(j_m)! M_2(j_m)!$~times in that sum.
This and \eqref{e:HOPproofHelmholtzCinftyStep22} yield
\begin{align*}
 \widehat{\beta}_{j_1,\dots,j_m}
   = (-1)^{\mathfrak{p}(j_{m})}
        (\mathrm{card}(\mathfrak{S}_{m-1}))^{-1}\, M_1(j_m)!\, M_2(j_m)!\,
      \sum_{\mathbf{j}\in\mathfrak{T}} \beta_{\mathbf{j}}.
\end{align*}
This reveals
\begin{align*}
 (\sym &\widehat{\beta})_{j_1,\dots,j_m}
   = (\mathrm{card}(\mathfrak{S}_m))^{-1} \sum_{\sigma\in\mathfrak{S}_m} 
              \widehat{\beta}_{j_{\sigma(1)},\dots,j_{\sigma(m)}}\\
   &= (\mathrm{card}(\mathfrak{S}_m))^{-1} (\mathrm{card}(\mathfrak{S}_{m-1}))^{-1}\\
  &\qquad\qquad\qquad  \times    \left(\sum_{\mathbf{j}\in\mathfrak{T}} \beta_{\mathbf{j}}\right)
         \sum_{\sigma\in\mathfrak{S}_m} 
          (-1)^{\mathfrak{p}(j_{\sigma(m)})}\,
          M_1(j_{\sigma(m)})!\, M_2(j_{\sigma(m)})!.
\end{align*}
A reordering of the summands and the definition of $M_1$ and $M_2$ 
in \eqref{e:HOPproofHelmholtzCinftyStep31} leads to 
\begin{align*}
  \sum_{\sigma\in\mathfrak{S}_m} 
          (-1)^{\mathfrak{p}(j_{\sigma(m)})}\,
      &    M_1(j_{\sigma(m)})!\, M_2(j_{\sigma(m)})!\\
    & = \Bigg(M_1(1)!\, M_2(1)! \sum_{\substack{\sigma\in\mathfrak{S}_m\\
                  j_{\sigma(m)}=1}} 1\Bigg)
        - \Bigg(M_1(2)!\, M_2(2)! \sum_{\substack{\sigma\in\mathfrak{S}_m\\
                  j_{\sigma(m)}=2}} 1\Bigg)\\
   & = (N_1-1)!\, N_2! \mathrm{card}(\{\sigma\in\mathfrak{S}_m\mid j_{\sigma(m)} = 1\})\\
   & \qquad\qquad- N_1!\, (N_2-1)! \mathrm{card}(\{\sigma\in\mathfrak{S}_m\mid j_{\sigma(m)}=2\}).
\end{align*}
Since $\mathrm{card}(\{\sigma\in\mathfrak{S}_m\mid j_{\sigma(m)} = 1\}) 
= N_1\, \mathrm{card}(\mathfrak{S}_{m-1})$ 
and $\mathrm{card}(\{\sigma\in\mathfrak{S}_m\mid j_{\sigma(m)}=2\}) 
= N_2\, \mathrm{card}(\mathfrak{S}_{m-1})$,
this vanishes. This proves $\sym\widehat{\beta}=0$.

\step{5}
Step~4 and $\tau\in C^\infty(\Omega;\Sym(m))$ leads to  
$\tau = \sym(\tau) = \sym(\tau -\widehat\beta)$.
Step~3 then yields $\tau = \sym\Curl\gamma$
and concludes the proof.
\end{proof}

The following theorem states a Helmholtz decomposition into $m$th derivatives 
and symmetric parts of Curls. The proof uses Theorem~\ref{t:HOPhelmholtzCinfty} 
and a density argument.
The following assumption assumes that the constant in 
Theorem~\ref{t:HOPcurllesssymcurl2} does continuously depend on the domain.
To this end, define 
\begin{align}\label{e:HOPdefOmegaeps}
 \Omega_{\varepsilon}:=\{x\in\Omega\mid 
\mathrm{dist}(x,\partial\Omega)>\varepsilon\}.
\end{align}

\begin{assumption}\label{as:HOPdepdomain}
There exist sequences $(\varepsilon_n)_{n\in\mathbb{N}}\in\R^\mathbb{N}$,
$(\delta_n)_{n\in\mathbb{N}}\in\R^\mathbb{N}$, 
and $(\Omega^{(n)})_{n\in\mathbb{N}}$ with 
$\Omega_{\delta_n}\subseteq\Omega^{(n)}\subseteq\Omega_{\varepsilon_n}\subseteq\Omega$
and $\varepsilon_n\to 0$ and $\delta_n\to 0$ as $n\to\infty$,
such that the constants $C_{n}$ from 
Theorem~\ref{t:HOPcurllesssymcurl2} with respect to 
$\Omega^{(n)}$
are uniformly bounded, $\sup_{n\in\mathbb{N}} C_{n} \lesssim 1$.
\end{assumption}

\begin{remark}
Remark~\ref{r:HOPdepdomain} implies that Assumption~\ref{as:HOPdepdomain} is 
fulfilled on star-shaped domains.
\end{remark}

Recall the definition of $Y$ from~\eqref{e:HOPdefY}.

\begin{theorem}[Helmholtz decomposition for higher-order derivatives]
\label{t:HOPhelmholtz}
If Assumption~\ref{as:HOPdepdomain} is satisfied, then it holds that 
\begin{align*}
 L^2(\Omega;\Sym(m))
  = D^m(H^m_0(\Omega)) 
    \oplus \sym\Curl Y
\end{align*}
and the decomposition is orthogonal in $L^2(\Omega;\Sym(m))$.
For any $\tau\in L^2(\Omega;\Sym(m))$, $u\in H^m_0(\Omega)$, and $\alpha\in Y$
with $\tau = D^m u + \sym\Curl \alpha$, the function $u\in H^m_0(\Omega)$
solves
\begin{align}\label{e:HOPhelmholtzSolutionProp}
 (D^m u,D^m v) = (\tau,D^m v)
 \qquad\text{for all }v\in H^m_0(\Omega).
\end{align}
\end{theorem}

\begin{proof}
Given $\tau\in L^2(\Omega;\Sym(m))$,
let $u\in H^m_0(\Omega)$ be the solution to \eqref{e:HOPhelmholtzSolutionProp}.
Define $r:=\tau-D^m u\in L^2(\Omega;\Sym(m))$ 
with $\ddiv^m r =0$.

Let $(\varepsilon_n)_{n\in\mathbb{N}}$, $(\delta_n)_{n\in\mathbb{N}}$, and 
$(\Omega^{(n)})_{n\in\mathbb{N}}$ denote the sequences from 
Assumption~\ref{as:HOPdepdomain} and let 
$\eta_n\in C^\infty_c(\R^2)$ denote the standard mollifier \cite{Evans2010} with 
compact support $\mathrm{supp}(\eta_n)$ in the ball $B_{\varepsilon_n}(0)$ with 
radius $\varepsilon_n$ and centre $0$.
Define the regularized function 
$r_n :=r*\eta_n\in C^\infty(\Omega;\Sym(m))$ with convolution $*$.
Then $r_n\to r$ in $L^2(\Omega;\Sym(m))$ as $n\to\infty$.
Recall the definition of $\Omega_{\varepsilon_n}$ from \eqref{e:HOPdefOmegaeps}. 
Since $\mathrm{supp}(\eta_n)\subseteq B_{\varepsilon_n}(0)$ and 
$\ddiv^m r=0$, it follows 
$(\ddiv^m r_n)\vert_{\Omega_{\varepsilon_n}} 
= (r * D^{m} \eta_n)\vert_{\Omega_{\varepsilon_n}} = 0$.
Since $\Omega^{(n)}\subseteq\Omega_{\varepsilon_n}$,
Theorem~\ref{t:HOPhelmholtzCinfty} guarantees the existence 
of $\gamma_n\in C^\infty(\Omega^{(n)};\mathbb{X}(m-1))$ with 
$r_n\vert_{\Omega^{(n)}} = \sym\Curl \gamma_n$.
Recall $\mathfrak{H}(m-1)$ from \eqref{e:HOPdefY} and define
\begin{align*}
  Z_n&:=\{\beta_n\in \mathfrak{H}(\Omega^{(n)},m-1)\mid 
     \sym\Curl\beta_n = 0\},\\
  Y_n&:=\{\zeta_n\in \mathfrak{H}(\Omega^{(n)},m-1)\mid
     \forall\beta_n \in Z_n:\; 
       (\Curl\beta_n,\Curl\zeta_n)_{L^2(\Omega)}=0\}.
\end{align*}
Let $\widetilde{\gamma}_n\in Y_n$ be the orthogonal 
projection (with respect to $(\Curl\bullet,\Curl\bullet)_{L^2(\Omega)}$) of 
$\gamma_n$ to $Y_n$. Then $\gamma_n - \widetilde{\gamma}_n\in Z_n$ and, hence, 
$\sym\Curl\widetilde{\gamma}_n = \sym\Curl\gamma_n = r_n\vert_{\Omega^{(n)}}$.
Let $\rho_n\in H^1(\Omega;\mathbb{X}(m-1))$ denote the extension of 
$\widetilde{\gamma}_n$ to $\Omega$ with $\|\rho_n\|_{H^1(\Omega)}
\lesssim\|\widetilde{\gamma}_n\|_{H^1(\Omega^{(n)})}$ \cite[Theorem~8.1]{LionsMagenes1972}. 
This, a Poincar\'e inequality, and Theorem~\ref{t:HOPcurllesssymcurl2} together 
with Assumption~\ref{as:HOPdepdomain} imply
\begin{align*}
 \left \|\rho_n\right\|_{H^1(\Omega)}
  \lesssim \left\|\Curl\widetilde{\gamma}_n\right\|_{L^2(\Omega^{(n)})}
  \lesssim \left\|\sym\Curl\widetilde{\gamma}_n\right\|_{L^2(\Omega^{(n)})}
  = \left\|r_n\right\|_{L^2(\Omega^{(n)})} \lesssim 1.
\end{align*}
Since $H^1(\Omega;\mathbb{X}(m-1))$ is reflexive,
there exists a subsequence of $(\rho_{n})_{n\in\mathbb{N}}$
(again denoted by $\rho_n$) and 
$\gamma\in H^1(\Omega;\mathbb{X}(m-1))$ with 
$\rho_n\rightharpoonup \gamma$ in 
$H^1(\Omega;\mathbb{X}(m-1))$.
Let $\varphi\in L^2(\Omega;\mathbb{X}(m))$ with 
$\mathrm{supp}(\varphi)\subseteq \Omega_{\delta_n}$. Since 
$\Omega_{\delta_n}\subseteq \Omega^{(n)}$ and therefore
$\sym\Curl\rho_n\vert_{\Omega_{\delta_n}}
=\sym\Curl\widetilde{\gamma}_n\vert_{\Omega_{\delta_n}} = 
r_n$, it follows
\begin{align*}
 (\varphi,\sym\Curl\gamma)_{L^2(\Omega)} 
   = (\varphi,r)_{L^2(\Omega)} + 
      (\varphi,\sym\Curl(\gamma-\rho_n))_{L^2(\Omega)} 
       + (\varphi,r_n-r)_{L^2(\Omega)}.
\end{align*}
Since $\rho_n\rightharpoonup\gamma$ in 
$H^1(\Omega;\mathbb{X}(m-1))$
and $r_n\to r$ in $L^2(\Omega;\Sym(m))$
and $\delta_n\to 0$, 
this leads to $\sym\Curl\gamma=r$.
Let $\rho\in Y$ be the orthogonal projection of $\gamma$ to $Y$ (with 
respect to $(\Curl\bullet,\Curl\bullet)_{L^2(\Omega)}$). Then 
$\rho-\gamma\in Z$ and, hence, $\sym\Curl\rho=\sym\Curl\gamma=r$.
This proves the decomposition.

Since $\Curl$ is the row-wise application of the standard Curl operator,
the $L^2$ orthogonality of $\Curl$ and $\nabla$ for scalar-valued 
functions and the symmetry of $D^m$ prove the $L^2$ orthogonality of 
$\sym\Curl$ and $D^m$.
\end{proof}

\section{Weak formulation and discretization}\label{s:HOPformDiscrete}

Subsection~\ref{ss:HOPweakform} introduces
the weak formulation of 
problem~\eqref{e:introHOP} based on the Helmholtz decomposition
from Section~\ref{s:HOPhelmholtzdecomposition} and its discretization
follows in Subsection~\ref{ss:HOPdiscretisation}.

\subsection{Weak formulation}\label{ss:HOPweakform}

Recall the definition of the divergence from Section~\ref{s:HOPnotation}
and the definition of $Y$ from~\eqref{e:HOPdefY}.
Let $\varphi\in H(\ddiv^m,\Omega)$ with $(-1)^m\ddiv^m\varphi=f$
and consider the problem:
Seek $(\sigma,\alpha)\in L^2(\Omega;\Sym(m))\times Y$
with 
\begin{equation}\label{e:HOPmixedproblem}
\begin{aligned}
 (\sigma,\tau)_{L^2(\Omega)} + (\tau,\sym\Curl \alpha)_{L^2(\Omega)} 
  &= (\varphi,\tau)_{L^2(\Omega)}
 &&\text{ for all }\tau\in L^2(\Omega;\Sym(m)),\\
 (\sigma,\sym\Curl \beta)_{L^2(\Omega)} &= 0 
 &&\text{ for all }\beta\in Y.
\end{aligned}
\end{equation}

The following theorem states the equivalence of this problem 
with \eqref{e:introHOP}.

\begin{theorem}[existence of solutions]\label{t:HOPexistence}
There exists a unique solution 
$(\sigma,\alpha)\in L^2(\Omega;\Sym(m))\times Y$ to \eqref{e:HOPmixedproblem}
with 
\begin{align}\label{e:HOPstability}
  \|\sigma\|_{L^2(\Omega)}^2 + \left\|\Curl\alpha\right\|_{L^2(\Omega)}^2 
 \lesssim
 \|\sigma\|_{L^2(\Omega)}^2 + \left\|\sym\Curl\alpha\right\|_{L^2(\Omega)}^2 
  = \left\|\sym\varphi\right\|_{L^2(\Omega)}^2.
\end{align}
If Assumption~\ref{as:HOPdepdomain} is satisfied, then $(\sigma,\alpha)$ satisfies 
$\sigma=D^m u$ for the solution $u\in H^m_0(\Omega)$ to
\eqref{e:introHOP}.
\end{theorem}

Note that $\sigma=D^m u$ is satisfied for any 
$\varphi\in H(\ddiv^m,\Omega)$ with $(-1)^m\ddiv^m\varphi=f$, while 
$\sym\Curl\alpha=\varphi-D^m u$ depends on the choice of $\varphi$.

\begin{proof}[Proof of Theorem~\ref{t:HOPexistence}]
The inf-sup condition 
\begin{align*}
 \left\|\Curl\beta\right\|_{L^2(\Omega)}
 \lesssim \sup_{\tau\in L^2(\Omega;\Sym(m))\setminus\{0\}} 
     \frac{(\tau,\sym\Curl\beta)_{L^2(\Omega)}}{\|\tau\|_{L^2(\Omega)}}
\end{align*}
follows from Theorem~\ref{t:HOPcurllesssymcurl2}. This and 
Brezzi's splitting lemma \cite{Brezzi1974} proves the unique 
existence of a solution 
to \eqref{e:HOPmixedproblem}. Since 
$\sigma+\sym\Curl\alpha=\sym(\varphi)$,
Theorem~\ref{t:HOPcurllesssymcurl2} leads to the stability 
\eqref{e:HOPstability}.

If Assumption~\ref{as:HOPdepdomain} is fulfilled, then 
the Helmholtz decomposition of Theorem~\ref{t:HOPhelmholtz} holds and 
the $L^2$ orthogonality of $\sigma$ to $\sym\Curl Y$ yields the existence 
of $\widetilde{u}\in H^m_0(\Omega)$ with $\sigma=D^m\widetilde{u}$.
The orthogonality of Theorem~\ref{t:HOPhelmholtz},  
$(-1)^m\ddiv^m\varphi=f$, and the symmetry of the $m$th derivative 
imply for all $v\in H^m_0(\Omega)$ that
\begin{align*}
  (D^m\widetilde{u},D^m v)_{L^2(\Omega)}
    = (\varphi,D^m v)_{L^2(\Omega)} - (D^m v, \sym\Curl\alpha)_{L^2(\Omega)}
    = (f,v).
\end{align*}
Hence, $\widetilde{u}$ solves \eqref{e:introHOP}.
\end{proof}

\subsection{Discretization}\label{ss:HOPdiscretisation}

The discretization of \eqref{e:HOPmixedproblem} employs the discrete 
spaces 
\begin{align*}
 X_h(\tri)&:=P_k(\tri;\Sym(m)),\\
 Y_h(\tri)&:=P_{k+1}(\tri;\mathbb{X}(m-1))\cap Y
\end{align*}
and seeks $\sigma_h\in X_h(\tri)$ and 
$\alpha_h\in Y_h(\tri)$
with
\begin{equation}\label{e:HOPdP}
\begin{aligned}
 (\sigma_h,\tau_h)_{L^2(\Omega)} + (\tau_h,\sym\Curl \alpha_h)_{L^2(\Omega)} 
   &= (\varphi,\tau_h)_{L^2(\Omega)}
 &&\text{ for all }\tau_h\in X_h(\tri),\\
 (\sigma_h,\sym\Curl \beta_h)_{L^2(\Omega)} &= 0 
 &&\text{ for all }\beta_h\in Y_h(\tri).
\end{aligned}
\end{equation}

\begin{remark}
Note that there is no constraint on the polynomial degree $k\geq 0$.
A discretization with the lowest polynomial degree 
involves only piecewise constant and piecewise affine 
functions for any $m\geq 1$. This should be contrasted to a standard conforming FEM 
where the $H^m_0(\Omega)$ conformity causes that the 
lowest possible polynomial degree is very high (cf.\ the Argyris FEM 
with piecewise $P_5$ functions and 21 local degrees of freedom for $m=2$ or 
the conforming FEM of \cite{Zenisek1970} for arbitrary $m$ with piecewise 
$P_{4(m-1)+1}$ functions).
Discontinuous 
Galerkin FEMs such as $C^0$ interior penalty methods 
\cite{EngelGarikipatiHughesLarsonMazzeiTaylor2002,Brenner2012} 
need at least piecewise $P_2$ functions for $m=2$ and piecewise $P_3$ 
functions for $m=3$ \cite{GudiNeilan2011}.
\end{remark}

\begin{remark}
Since the finite element spaces $X_h(\tri)$ and $Y_h(\tri)$ differ 
only in the number of components and the bilinear forms of 
\eqref{e:HOPdP} are similar for all $m$, an implementation in a single 
program which runs for all $m$ is possible. 
\end{remark}

\begin{remark}[Schur complement]
Since there is no continuity restriction in $X_h(\tri)$ between elements, 
the mass matrix is block diagonal with local mass matrices as sub-blocks.
Therefore, the matrix corresponding to the bilinear form 
$(\bullet,\bullet)_{L^2(\Omega)}$ in~\eqref{e:HOPdP} can be directly inverted.
\end{remark}

\begin{remark}
Problem~\eqref{e:HOPdP} provides an approximation $\sigma_h$ of $D^m u$.
If the function $u$ itself or a lower derivative of $u$ is the quantity of 
interest, it can be approximated by, e.g., a least squares approach. 
For $u_{h,m}:=\sigma_h$ the minimisation of 
\begin{align*}
  \sum_{j=0}^{m-1} \|u_{h,j+1} - D u_{h,j}\|_{L^2(\Omega)}^2
\end{align*}
with respect to $(u_{h,j})_{j=1,\dots,m-1}$ over a suitable finite element 
space results in a series of $m$ Poisson problems and provides an 
approximation $u_{h,0}$ to $u$. 
This ansatz can also be employed to include lower order terms in 
the system, cf.\ \cite{Gallistl2015} for a similar approach.
\end{remark}

\begin{theorem}[best-approximation result]\label{t:HOPbestapprox}
There exists a unique solution $(\sigma_h,\alpha_h)\in X_h(\tri)\times 
Y_h(\tri)$ to~\eqref{e:HOPdP} and it satisfies
\begin{equation}\label{e:HOPapriori}
\begin{aligned}
 & \|\sigma-\sigma_h\|_{L^2(\Omega)} 
    + \left\|\sym\Curl (\alpha-\alpha_h)\right\|_{L^2(\Omega)}\\
 &\qquad\lesssim \min_{\tau_h\in X_h(\tri)} \|\sigma-\tau_h\|_{L^2(\Omega)}
   +\min_{\beta_h\in Y_h(\tri)} 
     \left\|\sym\Curl (\alpha-\beta_h)\right\|_{L^2(\Omega)}.
\end{aligned}
\end{equation}
\end{theorem}

If the solution is sufficiently smooth, say $\sigma\in H^{k+1}(\Omega;\Sym(2))$
and $\alpha\in H^{k+2}(\Omega;\R^2)$, this yields a convergence rate 
of $\mathcal{O}(h^{k+1})$.

\begin{remark}[computation of $\varphi$]
Given a right-hand side $f$, the discretization~\eqref{e:HOPdP} requires the 
knowledge of a function $\varphi\in H(\ddiv^m,\Omega)$ with 
$(-1)^m\ddiv^m\varphi=f$. This can be computed by an integration of $f$ -- 
manually for a simple $f$ or numerically for a more complicated $f$. This 
can be done in parallel.
However, the numerical experiments of Section~\ref{s:HOPnumerics} and the 
best-approximation result in Theorem~\ref{t:HOPbestapprox}
suggest that the magnitude of the error heavily depends on the choice 
of $\varphi$ (which determines $\sym\Curl\alpha$). 
In Section~\ref{s:HOPnumerics}, the error can be drastically reduced by 
defining $\varphi$ by $\varphi=\Delta^{-1}\nabla\Delta^{-1}\nabla\Delta^{-1} f$
and approximate $\Delta^{-1}$ with standard finite elements 
(see Section~\ref{s:HOPnumerics} for more details).
\end{remark}

{\em Proof of Theorem~\ref{t:HOPbestapprox}.}
Since $\sym\Curl Y_h(\tri)\subseteq X_h(\tri)$, 
Theorem~\ref{t:HOPcurllesssymcurl2} proves
the inf-sup condition 
\begin{align*}
 \left\|\Curl\beta_h\right\|_{L^2(\Omega)}
  \lesssim \sup_{\tau_h\in X_h(\tri)\setminus\{0\}} 
     \frac{(\tau_h,\sym\Curl\beta_h)_{L^2(\Omega)}}{\|\tau_h\|_{L^2(\Omega)}}
  \qquad\text{for all }\beta_h\in Y_h(\tri).
\end{align*}
Brezzi's splitting lemma \cite{Brezzi1974} therefore
leads to the unique existence of a solution of problem \eqref{e:HOPdP}.
This, the conformity of the discretization, and standard arguments for mixed 
FEMs \cite{BoffiBrezziFortin2013} lead to the best-approximation 
result~\eqref{e:HOPapriori}.
\endproof

Define the space of \emph{discrete orthogonal derivatives} as
\begin{align}\label{e:HOPMh}
 W_h(\tri):=\{\tau_h\in X_h(\tri)\mid 
    \forall \beta_h\in Y_h(\tri):\; 
       (\tau_h,\sym\Curl\beta_h)_{L^2(\Omega)}=0\}.
\end{align}
The following lemma proves a projection property.
\begin{lemma}[projection property]\label{l:HOPintegralmean}
Let $\tau\in L^2(\Omega;\Sym(m))$ with 
\begin{align*}
 (\tau,\sym\Curl\beta)_{L^2(\Omega)}=0
  \qquad\text{for all }\beta\in Y.
\end{align*}
Then $\Pi_{X_h(\tri)}\tau\in W_h(\tri)$.
If $\tri_\star$ is 
an admissible refinement of $\tri$ and $\tau_\star\in W_h(\tri_\star)$, then 
$\Pi_{X_h(\tri)}\tau_\star\in W_h(\tri)$.
\end{lemma}

\begin{proof}
Let $\beta_h\in Y_h(\tri)$. Since 
$\sym\Curl\beta_h\in X_h(\tri)$, the conformity 
$Y_h(\tri)\subseteq Y$ implies 
\begin{equation*}
  (\Pi_{X_h(\tri)}\tau,\sym\Curl\beta_h)= (\tau,\sym\Curl\beta_h)=0.
\end{equation*}
The same arguments apply to $\tau_\star\in W_h(\tri_\star)$.
\end{proof}

\subsection{Application to Kirchhoff plates and the triharmonic equation}\label{ss:HOPplate}

For $m=2$, problem \eqref{e:introHOP} becomes the biharmonic 
problem $\Delta^2 u=f$.
This problem arises in the theory of Kirchhoff plates
with clamped boundary.
In this situation, the Helmholtz decomposition of Theorem~\ref{t:HOPhelmholtz}
is already proved in \cite{BeiraodaVeigaNiiranenStenberg2007}.

The discrete spaces in~\eqref{e:HOPdP} for $m=2$ read 
$X_h=P_k(\tri;\Sym(2))$ with $\Sym(2)$ the space of 
symmetric $2\times 2$ matrices and $Y_h=P_{k+1}(\tri;\R^2)\cap Y$ with $Y$
defined in~\eqref{e:HOPZm2}.
For plate bending problems, \cite{Morley1968} introduced a $P_2$ non-conforming 
finite element method (also called Morley FEM) with non-conforming finite element space
\begin{align*}
 V_M(\tri)&:=\left\{v_h\in P_2(\tri) \left| 
   \begin{array}{l}
     v_h\text{ is continuous at the interior nodes and vanishes at}\\
     \text{boundary nodes; }\nabla_\NC v_h\text{ is continuous at the interior}\\ 
      \text{edges' midpoints and vanishes at the midpoints of }\\
      \text{boundary edges}
   \end{array}
   \right. \right\}.
\end{align*}
The discrete Helmholtz decomposition \cite{CarstensenGallistlHu2014}
\begin{align*}
 P_0(\tri;\Sym(2)) = D_\NC^2 V_M(\tri) 
   \oplus \sym\Curl \big(P_1(\tri;\R^2)\cap Y\big).
\end{align*}
shows for $k=0$ the relation $D_\NC^2 V_M(\tri) = W_h(\tri)$
with $W_h(\tri)$ from \eqref{e:HOPMh} and, hence, the solution 
$\sigma_h$ to \eqref{e:HOPdP} is a piecewise Hessian of a 
Morley function.
If $\varphi$ satisfies $\ddiv^2\varphi=f$ also in the dual space of 
$V_M(\tri)$, then the solution $\sigma_h\in X_h(\tri)$ of~\eqref{e:HOPdP}
coincides with the piecewise Hessian of the solution of the Morley FEM.

For $m=3$, problem~\eqref{e:introHOP} becomes the triharmonic problem 
$-\Delta^3 u=f$. Sixth-order equations arise in the description of the motion of 
thin viscous droplets \cite{BarrettLangdonNuernberg2004}
or of the oxidation of silicon in superconductor devices \cite{King1989}.
For the triharmonic problem, the discrete spaces read $X_h=P_k(\tri;\Sym(3))$ and 
$Y_h=P_{k+1}(\tri;\R^{2\times 2})\cap Y$ with $Y$ defined 
in~\eqref{e:HOPdefY}. The orthogonality onto $Z$ implied by the definition 
of $Y$ can be implemented by Lagrange multipliers and with the 
knowledge of $Z$ from~\eqref{e:Zform=3}.

\section{Adaptive algorithm}\label{s:HOPafem}
This section defines the adaptive algorithm and proves its 
quasi-optimal convergence.

\subsection{Adaptive algorithm and optimal convergence rates}

Let $\tri_0$ denote some initial shape-regular triangulation of $\Omega$, such 
that each triangle $T\in\tri$ is equipped with a refinement edge 
$E_T\in\edges(T)$. We assume that $\tri_0$ fulfils the following initial 
condition.

\begin{definition}[initial condition]
All $T,K\in\tri_0$ with $T\cap K=E\in\edges$ and with refinement edges 
$E_T\in\edges(T)$ and $E_K\in\edges(K)$ satisfy: If $E_T=E$, then $E_K=E_T$. If 
$E_K=E$, then $E_T=E_K$. 
\end{definition}

Given an initial triangulation $\tri_0$,
the set of admissible triangulations $\mathbb{T}$ is 
defined as the set of all regular triangulations 
which can be created from $\tri_0$ by newest-vertex 
bisection (NVB) \cite{Stevenson2008}. 
Let $\mathbb{T}(N)$ denote the subset of all admissible 
triangulations with at most $\mathrm{card}(\tri_0) + N$ triangles.
The adaptive algorithm involves the overlay of two admissible triangulations
$\tri,\tri_\star\in\mathbb{T}$, which reads
\begin{align}\label{e:defoverlay}
  \tri\otimes\tri_\star:=\{T\in\tri\cup\tri_\star\mid 
      \exists K\in\tri,K_\star\in \tri_\star\text{ with }
       T\subseteq K\cap K_\star\}.
\end{align}

The adaptive algorithm is based on separate marking.
Given a triangulation $\tri_\ell$,
define for all $T\in\tri_\ell$ the local error estimator contributions by
\begin{align*}
  \lambda^2(\tri_\ell,T)
    &:= \|h_\tri \curl_\NC \sigma_h\|_{L^2(T)}^2
      + h_T \sum_{E\in \mathcal{E}(T)}  
         \left\|\left[\sigma_h\right]_E\cdot \tau_E \right\|_{L^2(E)}^2,\\
  \mu^2(T)&:= \left\|\sym(\varphi)-\Pi_k\sym(\varphi)\right\|_{L^2(T)}^2
\end{align*}
and the global error estimators by
\begin{align*}
 \lambda_\ell^2&:=\lambda^2(\tri_\ell,\tri_\ell)
 &\text{with }
 &&\lambda^2(\tri_\ell,\mathcal{M})&:=\sum_{T\in \mathcal{M}} \lambda^2(\tri_\ell,T)
 &\text{for all }\mathcal{M}\subseteq\tri_\ell,\\
 \mu^2_\ell&:=\mu^2(\tri_\ell)
 &\text{with }
 &&\mu^2(\mathcal{M})&:=\sum_{T\in\mathcal{M}} \mu^2(T)
 &\text{for all }\mathcal{M}\subseteq\tri_\ell.
\end{align*}
The adaptive algorithm is driven by these two error estimators
and runs the following loop.

\begin{algorithm}[AFEM for higher-order problems]
  \label{a:HOPafem}
\begin{algorithmic}
\Require Initial triangulation $\tri_0$, parameters $0<\theta_A\le 1$,
    $0<\rho_B<1$, $0<\kappa$.
\For{$\ell=0,1,2,\dots$}
 \State {\it Solve.}
 Compute solution $(\sigma_{\ell},\alpha_{\ell})\in X_h(\tri_\ell)\times Y_h(\tri_\ell)$ 
 of \eqref{e:HOPdP} with respect
 to $\tri_\ell$.
 \State {\it Estimate.}
 Compute estimator contributions 
 $\big(\lambda^2(\tri_\ell,T)\big)_{T\in\tri_\ell}$ 
 and 
 $\big(\mu^2(T)\big)_{T\in\tri_\ell}$.
 \If{$\mu_\ell^2\leq \kappa\lambda_\ell^2$}
 \State   
 {\it Mark.}
 The D\"orfler marking chooses a minimal subset 
 $\mathcal{M}_\ell\subseteq\tri_\ell$
 such that
 \State \qquad
 $
    \theta_A \lambda_\ell^2 
 \le  \lambda_\ell^2 (\tri_\ell,\mathcal{M}_\ell)
 $
 \State
 {\it Refine.}
 Generate the smallest admissible refinement
 $\tri_{\ell+1}$ of $\tri_\ell$ in which 
 \State \qquad at least all 
 triangles in $\mathcal{M}_\ell$ are refined.
 \Else 
 \State
   {\it Mark.} Compute an admissible triangulation $\tri\in\mathbb{T}$ with 
    $ \mu_\tri^2\leq \rho_B\mu_\ell^2$.
 \State 
   {\it Refine.} Generate the overlay $\tri_{\ell+1}$ of $\tri_\ell$ and $\tri$
    (cf.~\eqref{e:defoverlay}).
 \EndIf
\EndFor
\Ensure Sequence of triangulations
  $\left(\tri_\ell\right)_{\ell\in\mathbb N_0}$ and
  discrete solutions 
 $(\sigma_{\ell},\alpha_{\ell})_{\ell\in\mathbb{N}_0}$.
\hfill$\blacklozenge$
\end{algorithmic}
\end{algorithm}

The marking in the second case $\mu_\ell^2>\kappa\lambda_\ell^2$
can be realized by the algorithm \texttt{Approx} from 
\cite{CarstensenRabus,BinevDahmenDeVore2004}, i.e. 
the \emph{threshold second algorithm} \cite{BinevDeVore2004} 
followed by a completion algorithm.

For $s>0$ and $(\sigma,\alpha,\varphi)\in L^2(\Omega;\Sym(m))\times Y
\times H(\ddiv^m,\Omega)$, define
\begin{align*}
 \left| (\sigma,\alpha,\varphi)\right|_{\mathcal{A}_s}
   &:= \sup_{N\in\mathbb{N}_0} N^s 
      \inf_{\tri\in\mathbb{T}(N)} \Big( 
          \left\|\sigma-\Pi_{X_h(\tri)} \sigma \right\|_{L^2(\Omega)} \\ 
    &\qquad\qquad 
       + \inf_{\beta_\tri\in Y_h(\tri)} 
            \left\|\sym\Curl(\alpha-\beta_\tri)\right\|_{L^2(\Omega)}
           + \|\varphi - \Pi_{X_h(\tri)} \varphi\|_{L^2(\Omega)}\Big).
\end{align*}

\begin{remark}[pure local approximation class]
A ``row-wise'' application of \cite[Theorem~3.2]{Veeser2014}
proves 
\begin{align}
 \left| (\sigma,\alpha,\varphi)\right|_{\mathcal{A}_s}
  \approx \left| (\sigma,\alpha,\varphi)\right|_{\mathcal{A}_s'}
  &:= \sup_{N\in\mathbb{N}} N^s 
      \inf_{\tri\in\mathbb{T}(N)} \Big( 
          \|\sigma-\Pi_{X_h(\tri)} \sigma \|_{L^2(\Omega)}
  \notag\\
  &\qquad 
    + \left\|\sym(\Curl\alpha) 
           - \Pi_{X_h(\tri)}\sym(\Curl\alpha)\right\|_{L^2(\Omega)}
  \notag\\
  &\qquad    
    + \left\|\sym(\varphi) - \Pi_{X_h(\tri)} \sym(\varphi)\right\|_{L^2(\Omega)}
                                                        \Big).
 \tag*{\qed}
\end{align}
\end{remark}

In the following, we assume that the following axiom (B1) holds for the 
algorithm used in the step {\it Mark} for $\mu_\ell^2>\kappa\lambda_\ell^2$. 
For the algorithm~\texttt{Approx}, this 
assumption is a consequence of Axioms (B2) and (SA) from Subsection~\ref{ss:HOPaxiomB}
\cite{CarstensenRabus}.

\begin{assumption}[(B1) optimal data approximation]\label{as:HOPB1}
Assume that $\left| (\sigma,\alpha,\varphi)\right|_{\mathcal{A}_s}$ is finite.
Given a tolerance $\mathrm{Tol}$, 
the algorithm used in {\it Mark} 
in the second case ($\mu_\ell^2>\kappa\lambda_\ell^2$) 
in Algorithm~\ref{a:HOPafem}
computes $\tri_\star\in\mathbb{T}$ with 
\begin{align*}
 \mathrm{card}(\tri_\star) - \mathrm{card}(\tri_0)
  \lesssim \mathrm{Tol}^{-1/(2s)}
 \qquad\text{and}\qquad 
 \mu^2(\tri_\star)\leq \mathrm{Tol}.
\end{align*}
\end{assumption}

The following theorem states optimal convergence rates of  
Algorithm~\ref{a:HOPafem}.

\begin{theorem}[optimal convergence rates of AFEM]\label{t:HOPoptimalafem}
For $0<\rho_B<1$ and sufficiently small $0<\kappa$ and $0<\theta <1$, 
Algorithm~\ref{a:HOPafem}
computes sequences of triangulations $(\tri_\ell)_{\ell\in\mathbb{N}}$ 
and discrete solutions $(\sigma_\ell,\alpha_\ell)_{\ell\in\mathbb{N}}$ 
for the right-hand side $\varphi$
of optimal rate of convergence in the sense that 
\begin{align*}
 (\mathrm{card}(\tri_\ell) - \mathrm{card}(\tri_0))^s 
  \Big(\|\sigma-\sigma_\ell\|_{L^2(\Omega)} 
     + \left\|\sym\Curl(\alpha-\alpha_\ell)\right\|_{L^2(\Omega)}\Big)
   \lesssim \left| (\sigma,\alpha,\varphi)\right|_{\mathcal{A}_s}.
\end{align*}
\end{theorem}

The proof follows from the 
abstract framework of \cite{CarstensenRabus},
under the assumptions (A1)--(A4),
which are proved in 
Subsections~\ref{ss:HOPstabilityreduction}--\ref{ss:HOPquasiorthogonality},
the assumption (B1), which follows from (B2) and (SA) from 
Subection~\ref{ss:HOPaxiomB} below for the algorithm \texttt{Approx},
and efficiency of $\sqrt{\lambda^2+\mu^2}$, which follows from the standard 
bubble function technique of \cite{Verfuerth96}.

\subsection{(A1) stability and (A2) reduction}\label{ss:HOPstabilityreduction}

The following two theorems follow from the structure of the error estimator 
$\lambda$.

\begin{theorem}[stability]
Let $\tri_\star$ be an admissible refinement of $\tri$
and $\mathcal{M}\subseteq\tri\cap\tri_\star$.
Let $(\sigma_{\tri_\star},\alpha_{\tri_\star})\in X_h(\tri_\star)\times 
Y_h(\tri_\star)$
and $(\sigma_{\tri},\alpha_{\tri})\in X_h(\tri)\times Y_h(\tri)$
be the respective discrete solutions to \eqref{e:HOPdP}.
Then,
\begin{align*}
  \lvert \lambda(\tri_\star,\mathcal{M}) - 
\lambda(\tri,\mathcal{M})\rvert
   \lesssim \| \sigma_{\tri_\star} - \sigma_\tri\|_{L^2(\Omega)}.
\end{align*}
\end{theorem}

\begin{proof}
This follows with triangle inequalities,
inverse inequalities and the trace inequality from 
\cite[p.~282]{BrennerScott08} as in 
\cite[Proposition~3.3]{CasconKreuzerNochettoSiebert2008}.
\end{proof}

\begin{theorem}[reduction]
Let $\tri_\star$ be an admissible refinement of $\tri$.
Then there exists $0<\rho_2< 1$ and $\Lambda_2<\infty$ such that
\begin{align*}
 \lambda^2(\tri_\star,\tri_\star\setminus\tri)
  \leq \rho_2\lambda^2(\tri,\tri\setminus\tri_\star)
   + \Lambda_2 \|\sigma_{\tri_\star}- \sigma_\tri\|^2.
\end{align*}
\end{theorem}

\begin{proof}
This follows with a triangle inequality and the mesh-size 
reduction property $h_{\tri_\star}^2\vert_T\leq h_\tri^2\vert_T/2$
for all $T\in\tri_\star\setminus\tri$ as in 
\cite[Corollary~3.4]{CasconKreuzerNochettoSiebert2008}.
\end{proof}

\subsection{(A4) discrete reliability}\label{ss:HOPdrel}

The following theorem proves discrete reliability, i.e., 
the difference between two discrete solutions is bounded by the error 
estimators on refined triangles only.

\begin{theorem}[discrete reliability]\label{t:HOPdrel}
 Let $\tri_\star$ be an admissible refinement of $\tri$ with 
respective discrete solutions 
$(\sigma_{\tri_\star},\alpha_{\tri_\star})\in X_h(\tri_\star)\times 
Y_h(\tri_\star)$
and $(\sigma_{\tri},\alpha_{\tri})\in X_h(\tri)\times Y_h(\tri)$ 
of~\eqref{e:HOPdP}.
Then,
\begin{align*}
 \|\sigma_\tri - \sigma_{\tri_\star}\|^2
  + \left\|\sym\Curl(\alpha_\tri-\alpha_{\tri_\star})\right\|_{L^2(\Omega)}
 \lesssim \lambda^2(\tri,\tri\setminus\tri_\star)
      + \mu^2(\tri,\tri\setminus\tri_\star).
\end{align*}
\end{theorem}

\begin{proof}
Recall the definition of $W_h(\tri_\star)$ from~\eqref{e:HOPMh}.
Since $\sigma_\tri-\sigma_{\tri_\star}\in X_h(\tri_\star)$, there exist
$p_{\tri_\star}\in W_h(\tri_\star)$ and $r_{\tri_\star}\in Y_h(\tri_\star)$
with $\sigma_\tri-\sigma_{\tri_\star} 
= p_{\tri_\star} + \sym\Curl r_{\tri_\star}$.
The discrete error can be split as
\begin{align}\label{e:HOPdrelproof}
 \|\sigma_\tri-\sigma_{\tri_\star}\|_{L^2(\Omega)}^2
 = (\sigma_\tri-\sigma_{\tri_\star},p_{\tri_\star})_{L^2(\Omega)}
   + (\sigma_\tri-\sigma_{\tri_\star}, \sym\Curl r_{\tri_\star})_{L^2(\Omega)}.
\end{align}
The projection property, Lemma~\ref{l:HOPintegralmean},
proves $\Pi_{X_h(\tri)}p_{\tri_\star}\in W_h(\tri)$.
Hence, problem~\eqref{e:HOPdP} implies that the first term of the right-hand 
side equals 
\begin{align*}
 (\sigma_\tri-\sigma_{\tri_\star},p_{\tri_\star})_{L^2(\Omega)}
 = (\Pi_{X_h(\tri)}\varphi - \varphi,p_{\tri_\star})_{L^2(\Omega)}
 = (\Pi_{X_h(\tri)}\varphi - \Pi_{X_h(\tri_\star)}\varphi,
      p_{\tri_\star})_{L^2(\Omega)}.
\end{align*}
For any triangle $T\in\tri\cap\tri_\star$,
it holds $(\Pi_{X_h(\tri)}\varphi - \Pi_{X_h(\tri_\star)}\varphi)\vert_T=0$.
Since $\tri_\star$ is a refinement of $\tri$, this implies
\begin{align*}
 (\Pi_{X_h(\tri)}\varphi - \Pi_{X_h(\tri_\star)}\varphi,
      p_{\tri_\star})_{L^2(\Omega)}
 & \leq \|\Pi_{X_h(\tri)}\varphi 
       - \Pi_{X_h(\tri_\star)}\varphi\|_{\tri\setminus\tri_\star}
     \; \|p_{\tri_\star}\|_{L^2(\Omega)}\\
  &\leq \|\varphi - \Pi_{X_h(\tri)}\varphi\|_{\tri\setminus\tri_\star}
    \; \|p_{\tri_\star}\|_{L^2(\Omega)}.
\end{align*}

Let $r_\tri\in Y_h(\tri)$ denote 
the quasi interpolant from \cite{ScottZhang1990}
of $r_{\tri_\star}$ which satisfies the approximation 
and stability properties 
\begin{align*}
 \|h_\tri^{-1}(r_{\tri_\star}-r_\tri)\|_{L^2(\Omega)}
  + \|D(r_{\tri_\star}-r_\tri)\|_{L^2(\Omega)}
  \lesssim \|D r_{\tri_\star}\|_{L^2(\Omega)}
\end{align*}
and $(r_\tri)\vert_E = (r_{\tri_\star})\vert_E$ for all 
edges $E\in \edges(\tri)\cap\edges(\tri_\star)$.
Since $\sigma_\tri\in W_h(\tri)$ and 
$\sigma_{\tri_\star}\in W_h(\tri_\star)$, 
the symmetry of $\sigma_\tri$ implies
\begin{equation}\label{e:drelproof1}
\begin{aligned}
 (\sigma_\tri-\sigma_{\tri_\star}, \sym\Curl r_{\tri_\star})_{L^2(\Omega)}
  &= (\sigma_\tri, \sym\Curl (r_{\tri_\star}-r_\tri))_{L^2(\Omega)}\\
  &= (\sigma_\tri, \Curl (r_{\tri_\star}-r_\tri))_{L^2(\Omega)}.
\end{aligned}
\end{equation}
An integration by parts leads to 
\begin{align*}
 (\sigma_\tri, \Curl (r_{\tri_\star}-r_\tri))_{L^2(\Omega)}
  &= -(\curl_\NC \sigma_\tri, r_{\tri_\star}-r_\tri)_{L^2(\Omega)}\\
  &\qquad\qquad + \sum_{E\in\mathcal{E}(\tri)} 
         \int_E [\sigma_\tri\cdot\tau_E]_E (r_{\tri_\star}-r_\tri)\,ds.
\end{align*}
For a triangle $T\in \tri\cap\tri_\star$, any edge 
$E\in\edges(T)$ satisfies $E\in\edges(\tri)\cap\edges(\tri_\star)$
and, hence, $(r_\tri)\vert_T = (r_{\tri_\star})\vert_T$
for all $T\in\tri\cap\tri_\star$.
This, the Cauchy inequality, the approximation and 
stability properties of the quasi interpolant, and 
the trace inequality from \cite[p.~282]{BrennerScott08}
lead to 
\begin{equation}\label{e:drelproof2}
\begin{aligned}
 &-(\curl_\NC \sigma_\tri, r_{\tri_\star}-r_\tri)_{L^2(\Omega)}
 +\sum_{E\in\mathcal{E}} 
         \int_E [\sigma_\tri\cdot\tau_E]_E (r_{\tri_\star}-r_\tri)\,ds\\
 &\qquad
  \lesssim \bigg(\|h_\tri \curl_\NC \sigma_\tri\|_{\tri\setminus\tri_\star}\\
  &\qquad\qquad\qquad  
 + \sqrt{\sum_{E\in\edges(\tri)\setminus\edges(\tri_\star)}
       h_T \|[\sigma_\tri\cdot\tau_E]_E\|_{L^2(E)}^2} \bigg)
    \left\|\Curl r_{\tri_\star}\right\|_{L^2(\Omega)}.
\end{aligned}
\end{equation}
The combination of the previous displayed inequalities
yields 
\begin{align*}
 \|\sigma_\tri - \sigma_{\tri_\star}\|_{L^2(\Omega)}^2 
  \lesssim \lambda^2(\tri,\tri\setminus\tri_\star) 
+ \mu^2(\tri,\tri\setminus\tri_\star).
\end{align*}
Since $\Curl\alpha_\tri = \Pi_{X_h(\tri)}\varphi - \sigma_\tri$ and 
$\Curl\alpha_{\tri_\star} = \Pi_{X_h(\tri_\star)}\varphi - \sigma_{\tri_\star}$,
the triangle inequality yields the assertion.
\end{proof}

\begin{remark}[discrete reliability implies reliability]
The convergence of $\sigma_{\tri_\star}$ and $\alpha_{\tri_\star}$,
which is a consequence of the a~priori error estimate 
of Theorem~\ref{t:HOPbestapprox}, and the discrete reliability 
of Theorem~\ref{t:HOPdrel} imply the reliability
\begin{align*}
 \|\sigma-\sigma_\tri\|_{L^2(\Omega)}^2
  + \|\sym\Curl(\alpha - \alpha_\tri)\|_{L^2(\Omega)}^2
   \lesssim \lambda_\ell^2 + \mu_\ell^2.
\end{align*}
\end{remark}

\subsection{(A3) quasi-orthogonality}\label{ss:HOPquasiorthogonality}

The following theorem proves quasi-or\-thog\-o\-nal\-i\-ty of the 
discretization~\eqref{e:HOPdP}.

\begin{theorem}[general quasi-orthogonality]
Let $(\tri_j\mid j\in\mathbb{N})$ be some sequence of triangulations with discrete solutions  
$(\sigma_{j},\alpha_{j})\in X_h(\tri_j)\times Y_h(\tri_j)$
to \eqref{e:HOPdP} and let $\ell\in\mathbb{N}$. Then,
\begin{align*}
 \sum_{j=\ell}^\infty \Big( \|\sigma_j-\sigma_{j-1}\|^2 
          + \left\|\sym\Curl(\alpha_j-\alpha_{j-1})\right\|^2\Big)
    \lesssim \lambda_{\ell-1}^2 + \mu_{\ell-1}^2.
\end{align*}
\end{theorem}

\begin{proof}
The projection property, Lemma~\ref{l:HOPintegralmean}, proves 
$\Pi_{X_h(\tri_{j-1})} \sigma_j\in W_h(\tri_{j-1})$ with $W_h(\tri_{j-1})$ 
from~\eqref{e:HOPMh}. Hence, problem~\eqref{e:HOPdP} leads to
\begin{align*}
 (\sigma_{j-1},\sigma_j - \sigma_{j-1})_{L^2(\Omega)} 
    &= (\varphi,\Pi_{X_h(\tri_{j-1})} \sigma_j - \sigma_{j-1})_{L^2(\Omega)},\\
 (\sigma_j,\sigma_j - \sigma_{j-1})_{L^2(\Omega)} 
     &= (\varphi,\sigma_j) - (\varphi,\Pi_{X_h(\tri_{j-1})} \sigma_j)_{L^2(\Omega)}.
\end{align*}
The subtraction of these two equations and an index shift leads, for any 
$M\in\mathbb{N}$ with $M>\ell$, to 
\begin{equation}\label{e:PMPquasiorthogonalityProof1}
\begin{aligned}
 &\sum_{j=\ell}^M \|\sigma_j- \sigma_{j-1}\|_{L^2(\Omega)}^2
   = \sum_{j=\ell}^M 
    (\varphi, \sigma_j - \Pi_{X_h(\tri_{j-1})} \sigma_j)_{L^2(\Omega)} \\
  &\qquad\qquad\qquad\qquad\qquad\qquad
          - \sum_{j=\ell}^M (\varphi,\Pi_{X_h(\tri_{j-1})} \sigma_j)_{L^2(\Omega)}
          + \sum_{j=\ell-1}^{M-1} (\varphi,\sigma_{j})_{L^2(\Omega)}\\
  &\qquad\qquad\qquad
    = (\varphi,\sigma_{\ell-1} - \sigma_M)_{L^2(\Omega)} 
   + 2\sum_{j=\ell}^{M} 
      (\varphi, \sigma_j - \Pi_{X_h(\tri_{j-1})} \sigma_j)_{L^2(\Omega)} .
\end{aligned}
\end{equation}
Since $\sigma_j - \Pi_{X_h(\tri_{j-1})} \sigma_j\in X_h(\tri_j)$ 
is $L^2$-orthogonal to 
$X_h(\tri_{j-1})$, a Cauchy and a weighted Young inequality
imply
\begin{equation}\label{e:PMPquasiorthogonalityProof6}
\begin{aligned}
 &2\sum_{j=\ell}^{M} (\varphi, \sigma_j  
      - \Pi_{X_h(\tri_{j-1})} \sigma_j)_{L^2(\Omega)}\\
 &\quad= 2\sum_{j=\ell}^{M} 
   (\Pi_{X_h(\tri_j)}\varphi-\Pi_{X_h(\tri_{j-1})}\varphi, 
                     \sigma_j - \Pi_{X_h(\tri_{j-1})} \sigma_j)_{L^2(\Omega)}\\
 &\quad\leq 2 \sum_{j=\ell}^{M} 
     \|\Pi_{X_h(\tri_j)}\varphi-\Pi_{X_h(\tri_{j-1})}\varphi\|_{L^2(\Omega)}^2
   + \frac{1}{2} \sum_{j=\ell}^{M} 
         \|\sigma_j - \Pi_{X_h(\tri_{j-1})} \sigma_j\|_{L^2(\Omega)}^2.
\end{aligned}
\end{equation}
The orthogonality 
$\Pi_{X_h(\tri_j)}\varphi-\Pi_{X_h(\tri_{j-m})}\varphi\bot_{L^2(\Omega)} 
X_h(\tri_{j-m})$ for all $0\leq m\leq j$ 
and the definition of $\mu_\ell$ proves
\begin{equation}\label{e:PMPquasiorthogonalityProof3}
\begin{aligned}
 &\sum_{j=\ell}^{M}  
   \|\Pi_{X_h(\tri_j)}\varphi-\Pi_{X_h(\tri_{j-1})}\varphi\|_{L^2(\Omega)}^2 
   = \|\Pi_{X_h(\tri_M)} \varphi - 
                         \Pi_{X_h(\tri_{\ell-1})}\varphi\|_{L^2(\Omega)}^2\\
  &\qquad\qquad\qquad\qquad
  = \|\Pi_{X_h(\tri_M)} 
        (\varphi - \Pi_{X_h(\tri_{\ell-1})}\varphi)\|_{L^2(\Omega)}
 \leq \mu_{\ell-1}  .
\end{aligned}
\end{equation}
The combination of 
\eqref{e:PMPquasiorthogonalityProof1}--\eqref{e:PMPquasiorthogonalityProof3} and 
$\|\sigma_j-\Pi_{X_h(\tri_{j-1})}\sigma_j\|_{L^2(\Omega)}
\leq\|\sigma_j-\sigma_{j-1}\|_{L^2(\Omega)}$ leads to 
\begin{equation}\label{e:PMPquasiorthogonalityProof2}
\begin{aligned}
 \frac{1}{2}\sum_{j=\ell}^M \|\sigma_j- \sigma_{j-1}\|_{L^2(\Omega)}^2
   &\leq 2\mu_{\ell-1}^2
     + (\varphi,\sigma_{\ell-1} - \sigma_M)_{L^2(\Omega)} .
\end{aligned}
\end{equation}
The arguments of \eqref{e:drelproof1}--\eqref{e:drelproof2} prove 
\begin{align*}
 (\sym\Curl(\alpha_M-\alpha_{\ell-1}),\sigma_{\ell-1})_{L^2(\Omega)}
  \lesssim \lambda_{\ell-1} 
       \left\|\Curl(\alpha_M-\alpha_{\ell-1})\right\|_{L^2(\Omega)}.
\end{align*}
The discrete problem \eqref{e:HOPdP},  
the discrete reliability  
$\left\|\sym\Curl(\alpha_{M}-\alpha_{\ell-1})\right\|_{L^2(\Omega)}\lesssim 
\lambda_{\ell-1}+\mu_{\ell-1}$ from Theorem~\ref{t:HOPdrel}, and 
Theorem~\ref{t:HOPcurllesssymcurl2} therefore
lead to 
\begin{equation}\label{e:PMPquasiorthogonalityProof7}
\begin{aligned}
 &(\sigma_{\ell-1}- \sigma_{M},\Pi_{X_h(\tri_{\ell-1})}\varphi)_{L^2(\Omega)}
  = (\sigma_{\ell-1}-\sigma_{M},
          \sigma_{\ell-1} + \Curl\alpha_{\ell-1})_{L^2(\Omega)}\\
 &\qquad\qquad = (\sigma_{\ell-1}-\sigma_{M},\sigma_{\ell-1})_{L^2(\Omega)}
  = (\sym\Curl(\alpha_{M}-\alpha_{\ell-1}),\sigma_{\ell-1})_{L^2(\Omega)}\\
 &\qquad\qquad \lesssim \lambda_{\ell-1} 
    \left\|\Curl(\alpha_{M}-\alpha_{\ell-1})\right\|_{L^2(\Omega)}
   \lesssim \lambda_{\ell-1}^2+\mu_{\ell-1}^2.
\end{aligned}
\end{equation}
This and a further application of Theorem~\ref{t:HOPdrel} 
leads to 
\begin{equation}\label{e:PMPquasiorthogonalityProof8}
\begin{aligned}
 &(\varphi,\sigma_{\ell-1}-\sigma_{M})_{L^2(\Omega)} \\
 &\qquad\;
  = (\varphi-\Pi_{X_h(\tri_{\ell-1})}\varphi,
              \sigma_{\ell-1}-\sigma_{M})_{L^2(\Omega)} + 
   (\sigma_{\ell-1}-\sigma_{M},\Pi_{X_h(\tri_{\ell-1})}\varphi)_{L^2(\Omega)}\\
 &\qquad\;
  \lesssim \|\varphi-\Pi_{X_h(\tri_{\ell-1})}\varphi\|_{L^2(\Omega)}\;
     \|\sigma_{\ell-1}-\sigma_{M}\|_{L^2(\Omega)} 
    + (\lambda_{\ell-1}+\mu_{\ell-1})_{L^2(\Omega)}^2\\
  &\qquad\;
   \lesssim \lambda_{\ell-1}^2+\mu_{\ell-1}^2.
\end{aligned}
\end{equation}
The combination of \eqref{e:PMPquasiorthogonalityProof2} with 
\eqref{e:PMPquasiorthogonalityProof8} implies
\begin{align}\label{e:PMPquasiorthogonalityProof5}
 \sum_{j=\ell}^M \|\sigma_j- \sigma_{j-1}\|_{L^2(\Omega)}^2
    \lesssim \lambda_{\ell-1}^2+\mu_{\ell-1}^2.
\end{align}
The Young inequality,
the triangle inequality, and 
$\sym\Curl\alpha_j = \Pi_{X_h(\tri_j)} \varphi - \sigma_j$ 
imply
\begin{align*}
 &\sum_{j=\ell}^M
   \left\|\sym\Curl(\alpha_j - \alpha_{j-1})\right\|_{L^2(\Omega)}^2 \\
 &\qquad\qquad 
   \leq 2 \sum_{j=\ell}^M \|\sigma_j - \sigma_{j-1}\|_{L^2(\Omega)}^2
   + 2 \sum_{j=\ell}^M 
      \|\Pi_{X_h(\tri_j)}\varphi - 
                               \Pi_{X_h(\tri_{j-1})}\varphi\|_{L^2(\Omega)}^2.
\end{align*}
Since $M>\ell$ is arbitrary,
the combination with \eqref{e:PMPquasiorthogonalityProof3} 
and \eqref{e:PMPquasiorthogonalityProof5}
yields the assertion.
\end{proof}

\subsection{(B) data approximation}\label{ss:HOPaxiomB}

The following theorem states quasi-monotonicity and sub-additivity
for the data-ap\-prox\-ima\-tion error estimator $\mu$.
This theorem implies that Assumption~\ref{as:HOPB1} is satisfied 
if the algorithm \texttt{Approx} from \cite{BinevDeVore2004,
BinevDahmenDeVore2004,CarstensenRabus} is used in the second marking step 
($\mu_\ell^2\geq\kappa\lambda_\ell^2$)
in Algorithm~\ref{a:HOPafem} \cite{CarstensenRabus}.

\begin{theorem}[(B2) quasi-monotonicity and (SA) sub-additivity]
Any admissible refinement $\tri_\star$ of $\tri$ satisfies 
\begin{align*}
  \mu^2(\tri_\star) \leq \mu^2(\tri)
 \qquad\text{and}\qquad 
  \sum_{\substack{T\in\tri_\star\\
             T\subseteq K}}
   \mu^2(T) \leq \mu^2(K)
 \qquad\text{for all }K\in \tri.
\end{align*}
\end{theorem}

\begin{proof}
This follows directly from the definition of $\mu$.
\end{proof}

\section{Numerical experiments}\label{s:HOPnumerics}

This section is devoted to numerical experiments for the plate problem 
$\Delta^2u=f$ and the sixth-order problem $-\Delta^3 u=f$.
The discretization~\eqref{e:HOPdP} is realized for 
$k=0,1$ for the plate problem and for $k=0,1,2$ for the sixth-order problem. 
The experiments compare the errors and error 
estimators on a sequence of uniformly red-refined triangulations (that is, 
the midpoints of the edges of a triangle are connected;
this generates four new triangles) with the 
errors and error estimators on a sequence of triangulations created by
Algorithm~\ref{a:HOPafem} with bulk parameter $\theta=0.1$ and $\kappa=0.5$ 
and $\rho=0.75$.

The convergence history plots are logarithmically scaled and display the error 
$\|\sigma-\sigma_h\|_{L^2(\Omega)}$ against the number of degrees of freedom
(ndof) of the linear system resulting from the Schur complement.

\subsection{Square with known solution for $m=2$}
\label{ss:HOPnumSquare}

The exact solution to 
\begin{align*}
 \Delta^2 u (x,y)= f(x,y):= &24 (x^2-2 x^3 + x^4 + y^2-2 y^3 +y^4)\\
        &\qquad +2(2-12 x+12x^2) (2-12y+12y^2)
\end{align*}
with clamped boundary conditions 
$u\vert_{\partial\Omega}=(\partial u/\partial \nu)\vert_{\partial\Omega}=0$ 
reads
\begin{align*}
 u(x,y) = x^2 (1-x)^2 y^2 (1-y)^2.
\end{align*}
Define $\varphi=(\varphi_{jk})_{1\leq j,k\leq 2}\in H(\ddiv^2,\Omega)$ by
\begin{align*}
 \varphi_{11}&:= 24 (x^4/12 - x^5/10 + x^6/30) 
          + (x^2-2 x^3 +x^4) (2-12y+12y^2),\\
 \varphi_{22}&:= 24 (y^4/12 - y^5/10 + y^6/30) 
          + (y^2-2 y^3 +y^4) (2-12x+12x^2),\\
 \varphi_{12}&:=\varphi_{21}:=0.
\end{align*}
Then $\ddiv^2\varphi=f$ and $\varphi$ is an admissible right-hand side 
for \eqref{e:HOPdP}.

\begin{figure}
 \begin{center}
 \includegraphics{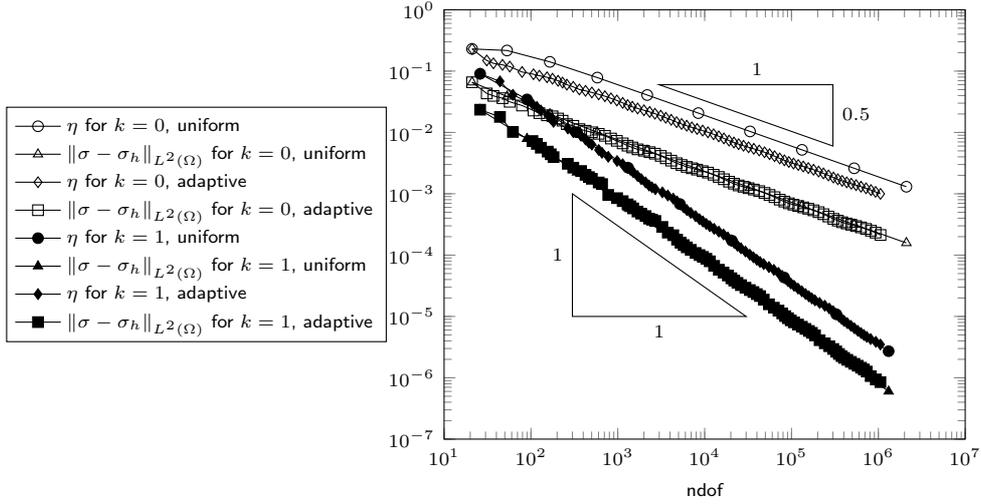}
 \end{center}
 \caption{\label{f:HOPnumSquare}Errors and error estimators for 
 the experiment on the square from Subsection~\ref{ss:HOPnumSquare}.}
\end{figure}

The errors $\|\sigma-\sigma_h\|_{L^2(\Omega)}$ and error estimators 
$\sqrt{\lambda^2+\mu^2}$ are plotted in Figure~\ref{f:HOPnumSquare}
versus the degrees of freedom. The errors and error 
estimators show an equivalent behaviour with an overestimation factor of 
approximately 10. The errors and error estimators show a convergence 
rate of $\texttt{ndof}^{-1/2}$ for $k=0$ and of $\texttt{ndof}^{-1}$ for $k=1$
on the sequence of uniformly red-refined triangulations as well as on the sequence 
of triangulations generated by Algorithm~\ref{a:HOPafem}. 
All marking steps in Algorithm~\ref{a:HOPafem} for 
$k=0,1$ applied the D\"orfler marking ($\mu_\ell^2\leq\kappa\lambda_\ell^2$).

\subsection{L-shaped domain with unknown solution for $m=2$}
\label{ss:HOPnumLshaped}

This subsection considers the problem 
\begin{align*}
 \Delta^2 u=1
\end{align*}
on the L-shaped domain $\Omega:=(-1,1)^2\setminus([0,1]\times[-1,0])$
with clamped boundary conditions 
$u\vert_{\partial\Omega}=(\partial u/\partial \nu)\vert_{\partial\Omega}=0$ 
and unknown solution.
Define the right-hand side 
$\varphi\in H(\ddiv^2,\Omega)$
with $\ddiv^2\varphi=1$ by
\begin{align*}
 \varphi(x,y):= \begin{pmatrix}
             x^2/4 & 0\\
              0 & y^2/4
           \end{pmatrix}.
\end{align*}

\begin{figure}
\begin{center}
\includegraphics{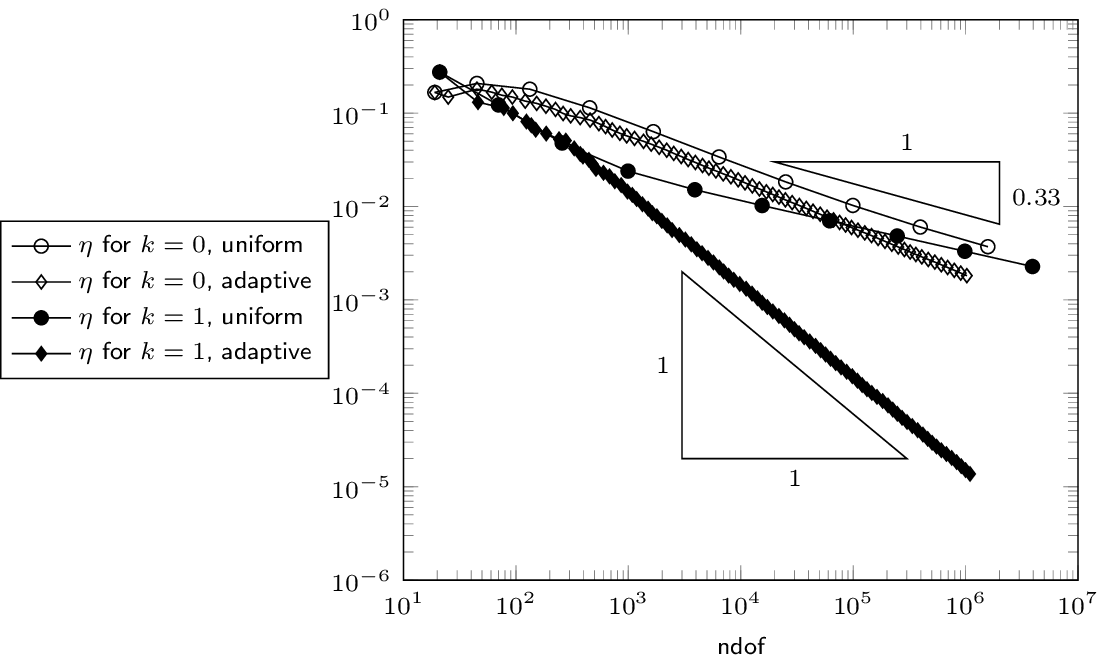}
\end{center}
 \caption{\label{f:HOPnumLshaped}Error estimators for 
 the experiment on the L-shaped domain from Subsection~\ref{ss:HOPnumLshaped}.}
\end{figure}

\begin{figure}
 \begin{center}
  \includegraphics[width=0.32\textwidth]
      {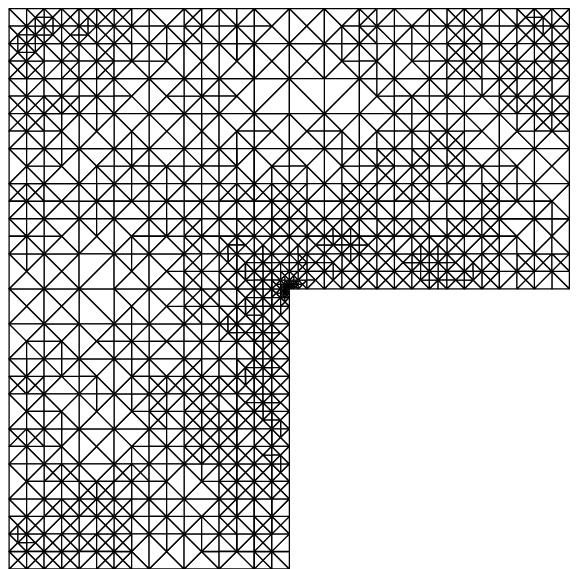}
  \includegraphics[width=0.32\textwidth]
      {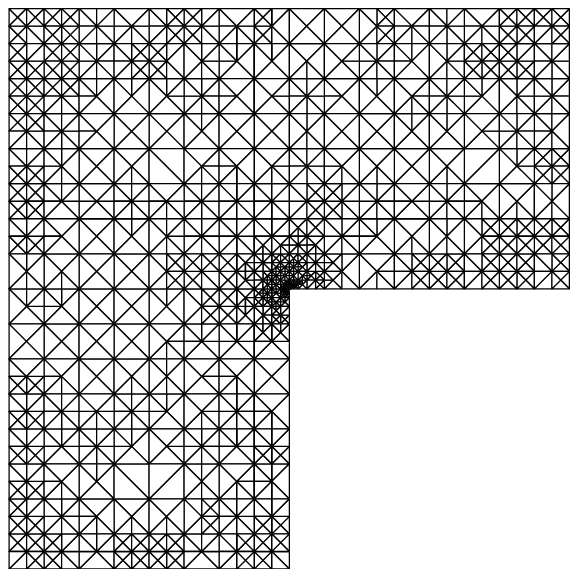}
 \end{center}
  \caption[Adaptively refined triangulations for the experiment on the L-shaped domain from 
Subsection~\ref{ss:HOPnumLshaped}.]{\label{f:HOPplateLshapedTriang}Adaptively refined 
triangulations for $k=0$ with 1096 nodes (2195 dofs) and for $k=1$ with 1077 
nodes (5114 dofs) for the experiment on the L-shaped domain from 
Subsection~\ref{ss:HOPnumLshaped}.}
\end{figure}

The error estimators $\sqrt{\lambda^2+\mu^2}$ 
are plotted in Figure~\ref{f:HOPnumLshaped}
versus the degrees of freedom. For uniform mesh-refinement the convergence 
rate of the error estimator for $k=1$ is $\texttt{ndof}^{-1/3}$. The convergence 
rate for $k=0$ is slightly larger, but the size of the error estimator is 
larger than for $k=1$. This suggests that the observed higher convergence rate 
is a preasymptotic effect.
On the sequences of triangulations generated by Algorithm~\ref{a:HOPafem},
the error estimators show the optimal convergence rates of $\texttt{ndof}^{-1/2}$ and 
$\texttt{ndof}^{-1}$ for $k=0$ and $k=1$, respectively. 
Figure~\ref{f:HOPplateLshapedTriang} displays triangulations with 
approximately 1000 vertices generated by Algorithm~\ref{a:HOPafem}
for $k=0$ and $k=1$. A stronger refinement towards the re-entrant corner
is clearly visible.
The marking with respect to the data-approximation 
($\mu_\ell^2>\kappa\lambda_\ell^2$ in Algorithm~\ref{a:HOPafem}) is only 
applied at the first two levels for $k=0$. All other marking steps for 
$k=0,1$ use the D\"orfler marking ($\mu_\ell^2\leq\kappa\lambda_\ell^2$).

\subsection{Square for $m=3$}\label{ss:HOPnumm3Square}

In this subsection, let $\Omega=(0,1)^2$ be the unit square and 
$u\in H^3_0(\Omega)$ be defined by
\begin{align*}
 u(x,y) = x^3 (1-x)^3 y^3 (1-y)^3
\end{align*}
with corresponding right-hand side $f:=-\Delta^3 u$.
Let $\varphi=(\varphi_{jk\ell})_{1\leq j,k,\ell\leq 2}\in H(\ddiv^3,\Omega)$
be defined by 
\begin{equation}\label{e:defRHSsquarem3}
\begin{aligned}
 \varphi_{111}(x,y)&:=-\frac{1}{2} \int_0^x \int_0^s \int_0^t f(\xi,y)\,d\xi\,dt\,ds,\\
 \varphi_{222}(x,y)&:=-\frac{1}{2}\int_0^y \int_0^s \int_0^t f(x,\xi)\,d\xi\,dt\,ds,\\
 \varphi_{112}&:=\varphi_{121}:=\varphi_{122}:=\varphi_{211}
   :=\varphi_{212}:=\varphi_{221}:=0.
\end{aligned}
\end{equation}
Then $-\ddiv^3\varphi=f$ and $\varphi$ is an admissible right-hand 
side for \eqref{e:HOPdP}.

\begin{figure}
\begin{center}
 \includegraphics{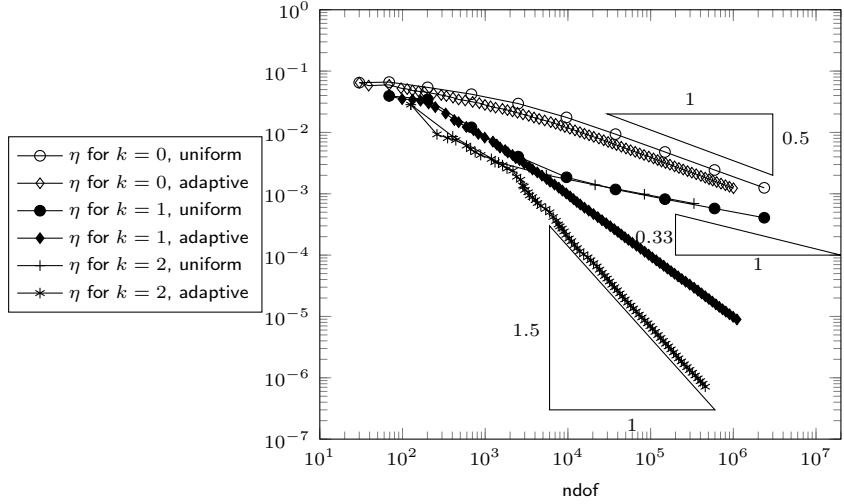}
\end{center}
 \caption{\label{f:HOPnumm3Square}Errors and error estimators for 
 the experiment on the square for $m=3$ from Subsection~\ref{ss:HOPnumm3Square}.
 The dashed lines correspond to the right-hand side generated 
 by the solution of three successive Poisson problems.}
\end{figure}

\begin{figure}
 \begin{center}
  \includegraphics[width=0.32\textwidth]
      {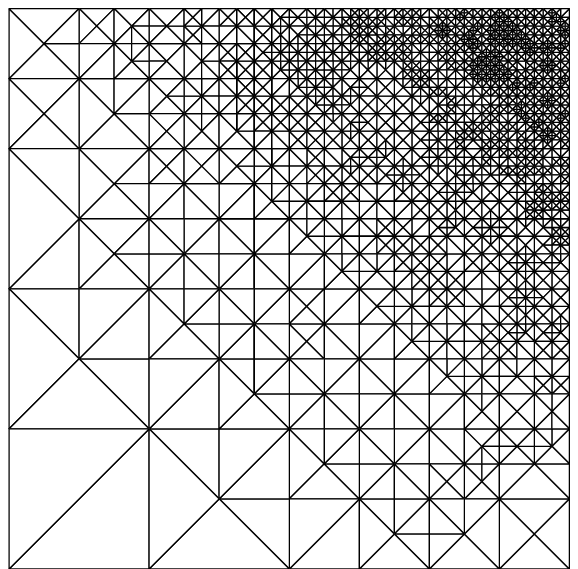}
  \includegraphics[width=0.32\textwidth]
      {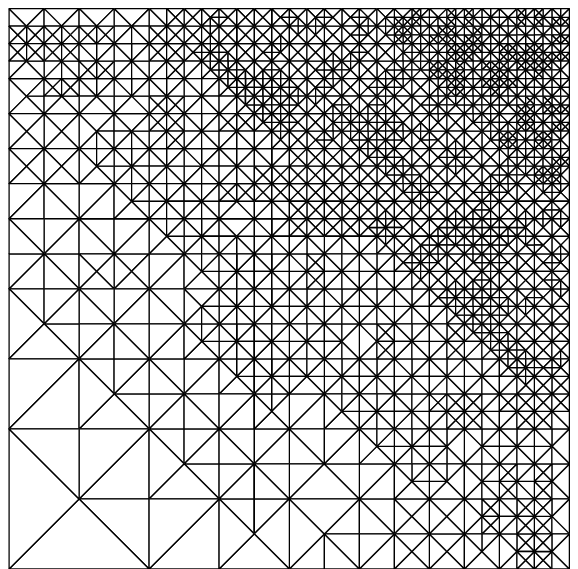}
  \includegraphics[width=0.32\textwidth]
      {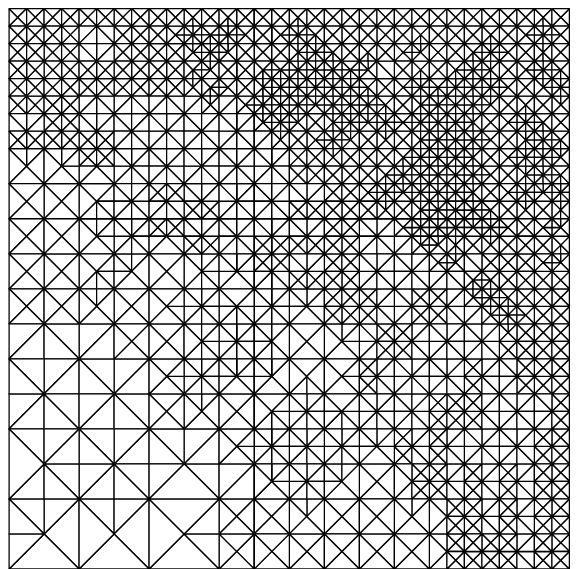}
 \end{center}
  \caption[Adaptively refined triangulations for the experiment on the square from 
Subsection~\ref{ss:HOPnumm3Square}.]{\label{f:HOPnumm3SquareTriang}Adaptively refined 
triangulations for $k=0$ with 1460 nodes (4386 dofs), for $k=1$ with 1555 
nodes (19369 dofs), and for $k=2$ with 1547 nodes (40833 dofs) 
for the experiment on the square from 
Subsection~\ref{ss:HOPnumm3Square}.}
\end{figure}

The errors $\|\sigma-\sigma_h\|_{L^2(\Omega)}$ and error estimators 
$\sqrt{\lambda^2+\mu^2}$ are plotted in Figure~\ref{f:HOPnumm3Square}
versus the number of degrees of freedom. 
The errors show the optimal convergence rates of $\texttt{ndof}^{-1/2}$, 
$\texttt{ndof}^{-1}$, and $\texttt{ndof}^{-3/2}$ for $k=0,1,2$ for uniform refinement
as well as for the sequence of triangulations generated by Algorithm~\ref{a:HOPafem}.
The error estimators for $k=0,1,2$ show an equivalent behaviour as the respective 
errors with an overestimation between 3 and 9. 

Although the convergence rates are optimal, one has to consider that the $H^3$-seminorm 
of the exact solution $\|\sigma\|_{L^2(\Omega)}$ is approximately $2\times 10^{-2}$.
That means that the relative errors for $k=1$ (resp.\ $k=2$) are larger than $100\%$ up to 
$10^5$ (resp.\ $10^4$) degrees of freedom and for $k=0$, they do not even reach 
this threshold.
While the $L^2$ norm of the function $\sigma$ of interest is approximately $10^{-2}$,
the $L^2$ norm of $\varphi$ (and thus $\left\|\Curl\alpha\right\|_{L^2(\Omega)}$)
is approximately 80. The best-approximation result~\eqref{e:HOPapriori} 
therefore seems to suffer from the large term 
\begin{align*}
 \inf_{\beta_h\in Y_h(\tri)}\left\|\Curl(\alpha-\beta_h)\right\|_{L^2(\Omega)}
\end{align*}
on the right-hand side.

A second choice for the right-hand side $\varphi$ should indicate one 
possibility to decrease the error. To this end, define 
$\widetilde{\varphi}:=\nabla w_3$ with
$(w_1,w_2,w_3)\in H^1_0(\Omega)\times H^1_0(\Omega;\R^2)\times H^1_0(\Omega;\R^{2\times 2})$
the solution of
\begin{equation}\label{e:computeRHSbyPMP}
\begin{aligned}
  (\nabla w_1,\nabla v)_{L^2(\Omega)} &= (f,v)_{L^2(\Omega)}
    &&\qquad\text{for all }v\in H^1_0(\Omega)\\
  (\nabla w_2,\nabla v)_{L^2(\Omega)} &= (\nabla w_1,v)_{L^2(\Omega)}
    &&\qquad\text{for all }v\in H^1_0(\Omega;\R^2)\\
  (\nabla w_3,\nabla v)_{L^2(\Omega)} &= (\nabla w_2,v)_{L^2(\Omega)}
    &&\qquad\text{for all }v\in H^1_0(\Omega;\R^{2\times 2}).
\end{aligned}
\end{equation}
Then $-\ddiv^3\widetilde{\varphi}=f$ and the computations are performed with 
the approximation $\widetilde{\varphi}_h$ of $\widetilde{\varphi}$ 
computed by the approximation of the Poisson problems~\eqref{e:computeRHSbyPMP}
by standard conforming FEMs of degree $k$.
The errors for this right-hand side are included in 
Figure~\ref{f:HOPnumm3Square} for $k=0,1,2$ with dashed lines.
The errors show the optimal convergence rates and the size of the errors are 
reduced by a factor between $10^2$ and $10^3$ compared to the errors for 
the right-hand side given by~\eqref{e:defRHSsquarem3}. In this situation,
the error is below $100\%$ for all triangulations.

Figure~\ref{f:HOPnumm3SquareTriang} displays triangulations with 
approximately 1500 vertices generated by Algorithm~\ref{a:HOPafem}
for $k=0,1,2$. Although the solution is smooth, a strong refinement towards 
the corner $(1,1)$ 
can be observed for $k=0$. For $k=1$, there is a slight refinement 
towards the corner $(1,1)$, while for $k=2$, the refinement is nearly uniform.
Since the relative errors for $k=0,1$ are still over $100\%$ on these 
triangulations, the discrete solution probably do not 
reflect the behaviour of the exact smooth solution.
However, the convergence rates are optimal and the error is slightly 
smaller compared with the uniform refinement. This is in agreement with 
Theorem~\ref{t:HOPoptimalafem}.

All marking steps in Algorithm~\ref{a:HOPafem} for 
$k=0,1,2$ used the D\"orfler marking ($\mu_\ell^2\leq\kappa\lambda_\ell^2$).

\subsection{L-shaped domain for $m=3$}
\label{ss:HOPnumm3Lshaped}

This section considers the problem: Find $u\in H^3_0(\Omega)$ with
\begin{align*}
 -\Delta^3 u=1
\end{align*}
and homogeneous Dirichlet boundary conditions on
the L-shaped domain 
$\Omega:=(-1,1)^2\setminus ([0,1]\times[-1,0])$.
Let $\varphi=(\varphi_{jk\ell})_{1\leq j,k,\ell\leq 2}\in H(\ddiv^3,\Omega)$
be defined by 
\begin{align*}
 \varphi_{111}(x,y)&:= -x^3/12,\\
 \varphi_{222}&:=-y^3/12,\\
 \varphi_{112}&:=\varphi_{121}:=\varphi_{122}:=\varphi_{211}
   :=\varphi_{212}:=\varphi_{221}:=0
\end{align*}
Then $-\ddiv^3\varphi=1$ and $\varphi$ is an admissible right-hand side for \eqref{e:HOPdP}.

\begin{figure}
\begin{center}
 \includegraphics{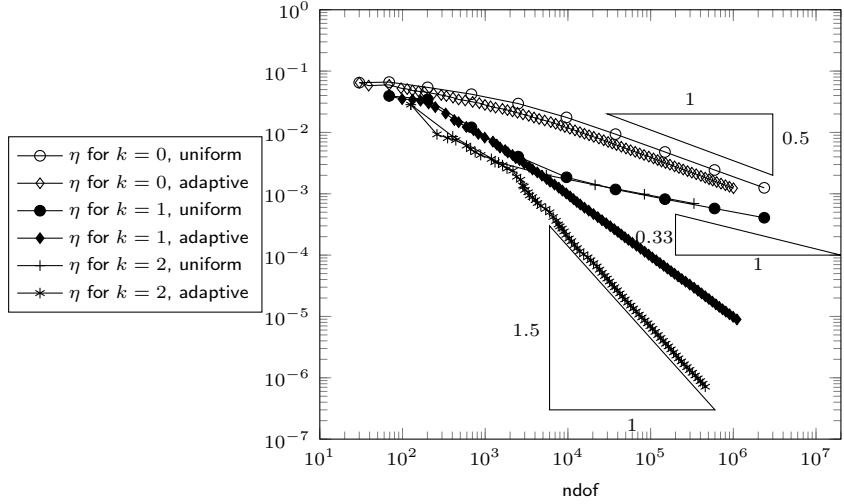}
\end{center}
 \caption{\label{f:HOPnumm3Lshaped}Errors and error estimators for 
 the experiment on the L-shaped domain for $m=3$ from Subsection~\ref{ss:HOPnumm3Lshaped}.}
\end{figure}

\begin{figure}
 \begin{center}
  \includegraphics[width=0.32\textwidth]
      {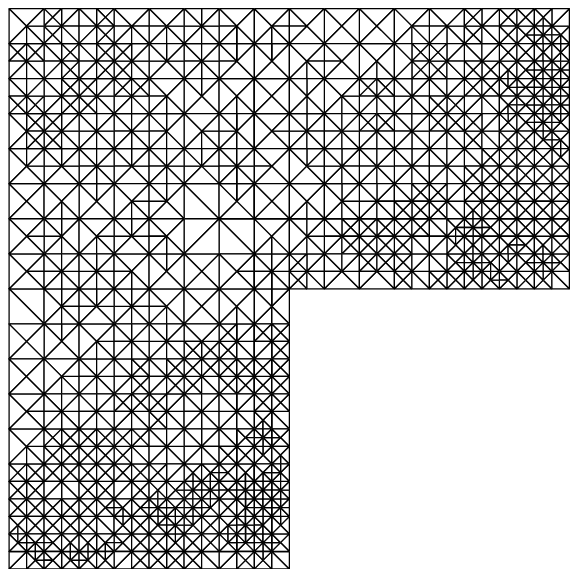}
  \includegraphics[width=0.32\textwidth]
      {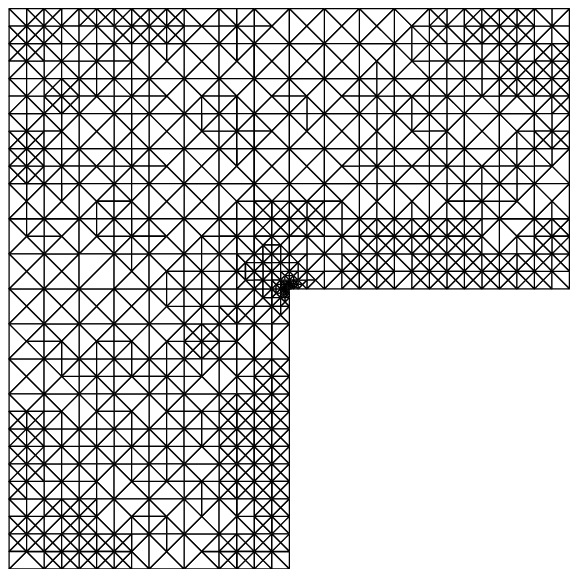}
  \includegraphics[width=0.32\textwidth]
      {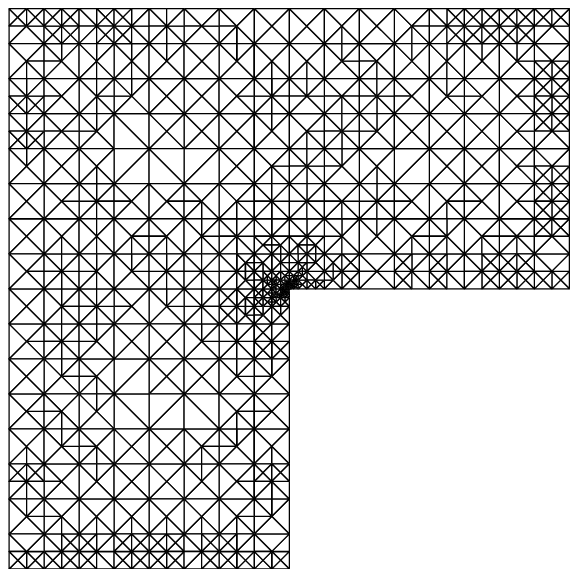}
 \end{center}
  \caption[Adaptively refined triangulations for the experiment on the L-shaped domain from 
Subsection~\ref{ss:HOPnumm3Lshaped}.]{\label{f:HOPnumm3LTriang}Adaptively refined 
triangulations for $k=0$ with 1118 nodes (3360 dofs), for $k=1$ with 1005 
nodes (11679 dofs), and for $k=2$ with 1004 nodes (25875 dofs) 
for the experiment on the L-shaped domain from 
Subsection~\ref{ss:HOPnumm3Lshaped}.}
\end{figure}

Since the exact solution is not known, only the error estimators 
$\sqrt{\lambda^2+\mu^2}$ are plotted in Figure~\ref{f:HOPnumm3Lshaped}
for $k=0,1,2$ on a sequence of uniformly red-refined triangulations and 
on a sequence generated by Algorithm~\ref{a:HOPafem}.
On the sequence of uniformly refined meshes, the error estimators for $k=1,2$ 
show a convergence rate of $\texttt{ndof}^{-1/3}$, while the error estimator 
for $k=0$ converges with rate $1/2$. However, this error 
estimator is of larger size than the error estimators for $k=1,2$ and 
it is therefore expected that the higher rate is a preasymptotic effect.
Algorithm~\ref{a:HOPafem} leads to the optimal convergence 
rates of $\texttt{ndof}^{-1/2}$ for $k=0$, $\texttt{ndof}^{-1}$ for $k=1$, 
and $\texttt{ndof}^{-3/2}$ for $k=2$.

Figure~\ref{f:HOPnumm3LTriang} displays triangulations with 
approximately 1000 vertices generated by Algorithm~\ref{a:HOPafem}
for $k=0,1,2$. The strong refinement towards the re-entrant corner is 
clearly visible for $k=1,2$, while for $k=0$ the refinement is quasi-uniform.
This is in agreement with the observed convergence rate for $k=0$ and the 
interpretation that the behaviour of the exact solution is not reflected 
in the discrete solution up to this number of degrees of freedom.
The marking with respect to the data-approximation 
($\mu_\ell^2>\kappa\lambda_\ell^2$ in Algorithm~\ref{a:HOPafem}) is only 
applied at levels 1 and 2 for $k=0$. All other marking steps for 
$k=0,1,2$ use the D\"orfler marking ($\mu_\ell^2\leq\kappa\lambda_\ell^2$).

\newcommand{\etalchar}[1]{$^{#1}$}


\begin{thebibliography}{EGH{\etalchar{+}}02}

\bibitem[AF89]{ArnoldFalk1989}
D.~N. Arnold and R.~S. Falk.
\newblock A uniformly accurate finite element method for the
  {R}eissner-{M}indlin plate.
\newblock {\em SIAM J. Numer. Anal.}, 26(6):1276--1290, 1989.

\bibitem[BBF13]{BoffiBrezziFortin2013}
D.~Boffi, F.~Brezzi, and M.~Fortin.
\newblock {\em Mixed Finite Element Methods and Applications}, volume~44 of
  {\em Springer Series in Computational Mathematics}.
\newblock Springer, Heidelberg, 2013.

\bibitem[BD04]{BinevDeVore2004}
P.~Binev and R.~DeVore.
\newblock Fast computation in adaptive tree approximation.
\newblock {\em Numer. Math.}, 97(2):193--217, 2004.

\bibitem[BDD04]{BinevDahmenDeVore2004}
P.~Binev, W.~Dahmen, and R.~DeVore.
\newblock Adaptive finite element methods with convergence rates.
\newblock {\em Numer. Math.}, 97(2):219--268, 2004.

\bibitem[BLN04]{BarrettLangdonNuernberg2004}
J. W. Barrett, S. Langdon, and R. N{\"u}rnberg.
\newblock Finite element approximation of a sixth order nonlinear degenerate
  parabolic equation.
\newblock {\em Numer. Math.}, 96(3):401--434, 2004.

\bibitem[BNS07]{BeiraodaVeigaNiiranenStenberg2007}
L.~{Beir{\~a}o da Veiga}, J.~Niiranen, and R.~Stenberg.
\newblock A posteriori error estimates for the {M}orley plate bending element.
\newblock {\em Numer. Math.}, 106(2):165--179, 2007.

\bibitem[Bre74]{Brezzi1974}
F.~Brezzi.
\newblock On the existence, uniqueness and approximation of saddle-point
  problems arising from {L}agrangian multipliers.
\newblock {\em Rev. Fran\c caise Automat. Informat. Recherche Op\'erationnelle
  S\'er. Rouge}, 8(R-2):129--151, 1974.

\bibitem[Bre12]{Brenner2012}
S.~C. Brenner.
\newblock {$C^0$} interior penalty methods.
\newblock In {\em Frontiers in Numerical Analysis---{D}urham 2010}, volume~85
  of {\em Lect. Notes Comput. Sci. Eng.}, pages 79--147. Springer, Heidelberg,
  2012.

\bibitem[BS08]{BrennerScott08}
S.~C. Brenner and L.~R. Scott.
\newblock {\em The Mathematical Theory of Finite Element Methods}, volume~15 of
  {\em Texts in Applied Mathematics}.
\newblock Springer Verlag, New York, Berlin, Heidelberg, 3 edition, 2008.

\bibitem[CGH14]{CarstensenGallistlHu2014}
C.~Carstensen, D.~Gallistl, and J.~Hu.
\newblock A discrete {H}elmholtz decomposition with {M}orley finite element
  functions and the optimality of adaptive finite element schemes.
\newblock {\em Comput. Math. Appl.}, 68(12):2167--2181, 2014.

\bibitem[Cia78]{Ciarlet1978}
Ph.~G. Ciarlet.
\newblock {\em The Finite Element Method for Elliptic Problems}.
\newblock Studies in Mathematics and its Applications, Vol. 4. North-Holland
  Publishing Co., Amsterdam-New York-Oxford, 1978.

\bibitem[CKNS08]{CasconKreuzerNochettoSiebert2008}
J.~M. Cascon, Ch. Kreuzer, R.~H. Nochetto, and K.~G. Siebert.
\newblock Quasi-optimal convergence rate for an adaptive finite element method.
\newblock {\em SIAM J. Numer. Anal.}, 46(5):2524--2550, 2008.

\bibitem[CR73]{CrouzeixRaviart1973}
M.~Crouzeix and P.-A. Raviart.
\newblock Conforming and nonconforming finite element methods for solving the
  stationary {S}tokes equations. {I}.
\newblock {\em Rev. Fran\c caise Automat. Informat. Recherche Op\'erationnelle
  S\'er. Rouge}, 7(R-3):33--75, 1973.

\bibitem[CR15]{CarstensenRabus}
C.~Carstensen and H.~Rabus.
\newblock Axioms of adaptivity for separate marking.
\newblock preprint, arXiv:1606.02165, 2016.

\bibitem[EGH{\etalchar{+}}02]{EngelGarikipatiHughesLarsonMazzeiTaylor2002}
G.~Engel, K.~Garikipati, T.~J.~R. Hughes, M.~G. Larson, L.~Mazzei, and R.~L.
  Taylor.
\newblock Continuous/discontinuous finite element approximations of
  fourth-order elliptic problems in structural and continuum mechanics with
  applications to thin beams and plates, and strain gradient elasticity.
\newblock {\em Comput. Methods Appl. Mech. Engrg.}, 191(34):3669--3750, 2002.

\bibitem[Eva10]{Evans2010}
L.~C. Evans.
\newblock {\em Partial Differential Equations}, volume~19 of {\em Graduate
  Studies in Mathematics}.
\newblock American Mathematical Society, Providence, RI, second edition, 2010.

\bibitem[Gal15]{Gallistl2015}
D.~Gallistl.
\newblock Stable splitting of polyharmonic operators by generalized {S}tokes
  systems.
\newblock {\em INS Preprint 1529}, Institut f\"ur Numerische Simulation, Germany, 2015.

\bibitem[GN11]{GudiNeilan2011}
Th. Gudi and M.~Neilan.
\newblock An interior penalty method for a sixth-order elliptic equation.
\newblock {\em IMA J. Numer. Anal.}, 31(4):1734--1753, 2011.

\bibitem[Kin89]{King1989}
J.~R. King.
\newblock The isolation oxidation of silicon: the reaction-controlled case.
\newblock {\em SIAM J. Appl. Math.}, 49(4):1064--1080, 1989.

\bibitem[LM72]{LionsMagenes1972}
J.-L. Lions and E.~Magenes.
\newblock {\em Non-homogeneous Boundary Value Problems and Applications. {V}ol.
  {I}}.
\newblock Die Grundlehren der mathematischen Wissenschaften, Band 181.
  Springer-Verlag, New York-Heidelberg, 1972.

\bibitem[Mor68]{Morley1968}
L.S.D. Morley.
\newblock The triangular equilibrium element in the solution of plate bending
  problems.
\newblock {\em Aeronaut.Quart.}, 19:149--169, 1968.

\bibitem[Ne{\v{c}}67]{Necas1967}
J.~Ne{\v{c}}as.
\newblock {\em Les M\'ethodes Directes en Th\'eorie des \'Equations
  Elliptiques}.
\newblock Masson et Cie, \'Editeurs, Paris; Academia, \'Editeurs, Prague, 1967.

\bibitem[Rud76]{Rudin1976}
W.~Rudin.
\newblock {\em Principles of Mathematical Analysis}.
\newblock McGraw-Hill Book Co., New York-Auckland-D\"usseldorf, third edition,
  1976.

\bibitem[Sch15]{Schedensack2015}
M.~Schedensack.
\newblock A new generalization of the ${P}_1$ non-conforming {FEM} to higher
  polynomial degrees.
\newblock 2015.
\newblock Preprint, arXiv:1505.02044.

\bibitem[Sch16]{Schedensack2016}
M.~Schedensack.
\newblock Mixed finite element methods for linear elasticity and the Stokes equations
based on the Helmholtz decomposition.
\newblock {\em ESAIM Math. Model. Numer. Anal.}, http://dx.doi.org/10.1051/m2an/2016024,
2016.

\bibitem[Ste08]{Stevenson2008}
R.~Stevenson.
\newblock The completion of locally refined simplicial partitions created by
  bisection.
\newblock {\em Math. Comp.}, 77(261):227--241, 2008.

\bibitem[SZ90]{ScottZhang1990}
L.~R. Scott and S.~Zhang.
\newblock Finite element interpolation of nonsmooth functions satisfying
  boundary conditions.
\newblock {\em Math. Comp.}, 54(190):483--493, 1990.

\bibitem[Vee14]{Veeser2014}
A.~Veeser.
\newblock Approximating gradients with continuous piecewise polynomial
  functions.
\newblock {\em Foundations of Computational Mathematics}, pages 1--28, 2014.

\bibitem[Ver96]{Verfuerth96}
R.~Verf{\"u}rth.
\newblock {\em A Review of a Posteriori Error Estimation and Adaptive
  Mesh-Refinement Techniques}.
\newblock Advances in numerical mathematics. Wiley, 1996.

\bibitem[WX13]{WangXu2013}
M. Wang and J. Xu.
\newblock Minimal finite element spaces for {$2m$}-th-order partial
  differential equations in {$R^n$}.
\newblock {\em Math. Comp.}, 82(281):25--43, 2013.

\bibitem[{\v{Z}}en70]{Zenisek1970}
A.~{\v{Z}}en{\'{\i}}{\v{s}}ek.
\newblock Interpolation polynomials on the triangle.
\newblock {\em Numer. Math.}, 15:283--296, 1970.

\end{thebibliography}
\end{document}